\documentclass[11pt]{amsart}
\usepackage{graphicx,amssymb,amsmath,amsthm}
\usepackage{ulem}
\usepackage{enumerate}
\usepackage{comment,cite,color}
\usepackage{cite,color}
\usepackage{mathrsfs}
\usepackage{epsfig}
\usepackage{subfigure}
\usepackage{lscape}
\usepackage{subfigure}
\usepackage{epstopdf}
\usepackage{array}
\usepackage{caption}
\usepackage{algorithm}
\usepackage{algpseudocode}
\usepackage{bm}
\usepackage{appendix}
\newcommand{\PP}{\mathbb{P}}

\newcommand{\abs}[1]{\lvert#1\rvert}

\renewcommand{\Re}{{\rm Re}}
\renewcommand{\Im}{{\rm Im}}

\newcommand{\0}{{\mathbf 0}}

\newcommand{\dist}{{\textup{dist}}}

\newcommand{\C}{{\mathbb C}}

\renewcommand{\eqref}[1]{(\ref{#1})}

\newcommand{\E}{{\mathbb E}}

\newcommand{\NN}{{\mathcal N}}

\newcommand{\vx}{{\mathbf x}}

\newcommand{\vz}{{\mathbf z}}

\newcommand{\va}{{\mathbf a}}
\newcommand{\vb}{{\mathbf b}}

\newcommand{\vv}{{\mathbf v}}
\newcommand{\vu}{{\mathbf u}}

\newcommand{\vh}{{\mathbf h}}
\newcommand{\vepsilon}{{\bm \epsilon}}

\newcommand{\s}{{\mathbf s}}

\newtheorem{theorem}{Theorem}[section]

\newtheorem{lemma}{Lemma}[section]
\newtheorem{remark}{Remark}[section]

\newcommand{\zz}{^{\top}}
\usepackage{latexsym}
\makeatletter
\@namedef{subjclassname@2020}{\textup{2020} Mathematics Subject Classification}
\makeatother

\usepackage[top=1in,bottom=1in,left=1.1in,right=1.1in]{geometry}
\date{}

\begin{document}
	\title{Convergence analysis of Wirtinger Flow for Poisson phase retrieval}	
	\author{Bing Gao}
	\address{Department of Mathematics, Nankai University, Tianjin, 300071, China}
	\email{gaobing@nankai.edu.cn}
	\author{Ran Gu}
	\address{NITFID, School of Statistics and Data Science, Nankai University, Tianjin 300071, China}
    \email{rgu@nankai.edu.cn}
    \thanks{The second author is the corresponding author.}
    \author{Shigui Ma}
    \address{College of Tourism and Service Management, Nankai University, Tianjin 300071, China}
    \email{sma@nankai.edu.cn}
    \thanks{Bing Gao was supported by NSFC grant \#12001297. Ran Gu was supported in part by National Key R\&D Program of China grant \#2022YFA1003800, NSFC grant \#12201318 and the Fundamental Research Funds for the Central Universities \#63223078. Shigui Ma was supported by NSFC grant \#72301147.}
\maketitle
\begin{abstract}
	This paper presents a rigorous theoretical convergence analysis of the Wirtinger Flow (WF) algorithm for Poisson phase retrieval, a fundamental problem in imaging applications. Unlike prior analyses that rely on truncation or additional adjustments to handle outliers, our framework avoids eliminating measurements or introducing extra computational steps, thereby reducing overall complexity. We prove that WF achieves \textbf{linear convergence} to the true signal under noiseless conditions and remains \textbf{robust and stable} in the presence of bounded noise for Poisson phase retrieval. Additionally, we propose an incremental variant of WF, which significantly improves computational efficiency and guarantees convergence to the true signal with high probability under suitable conditions.
\end{abstract}
\section{Introduction}
\subsection{Problem setup and related work}
Phase retrieval is a fundamental computational problem with diverse applications, including optics\cite{walther1963question}\cite{millane1990phase}, X-ray crystallography\cite{miao1999extending}\cite{miao2008extending}, and astronomical imaging\cite{fienup1987phase}. It involves recovering a signal from intensity-only measurements, which arise due to inherent physical limitations. Formally, the goal is to recover a signal $\vx$ from intensity-only measurements $\{y_j = |\va_j^*\vx|^2,\,j=1,\ldots,m\}$, where the phase information is inherently lost. Here, $\va_j$ denotes a sampling vector, which could be Fourier transforms, coded diffraction patterns, short-time Fourier transforms or random Gaussian transforms.

Classical algorithms for phase retrieval are based on alternating projection methods, such as the Gerchberg-Saxton algorithm\cite{gerchberg1972practical}, the Fienup algorithm\cite{fienup1982phase}, and alternating minimization\cite{netrapalli2013phase}. While simple to implement and parameter-free, these methods are often hindered by the non-convexity of the problem, which can lead to convergence to local minima. To address these limitations, convex relaxation approaches like PhaseLift\cite{candes2013phaselift} and PhaseCut\cite{waldspurger2015phase} reformulate the phase retrieval problem by “lifting” the signal $ \vx $ to a rank-one matrix $ \vx\vx^* $ and similarly lifting $ \va_j $ as $ \va_j\va_j^* $. Then the observations are actually linear measurements of the matrix $ \vx\vx^*\in\C^{n\times n}$.  When the number of measurements $ m $ is sufficiently large, these convex algorithms are robust to noise. However, the lifting operation significantly increases computational complexity, limiting their efficiency for large-scale problems.
 
In recent years, numerous non-convex algorithms have been developed to bypass the complexity of lifting while efficiently achieving global optima. These algorithms typically involve two stages: initialization and refinement. In the 
refinement stage, the objective function to be minimized is often chosen as:
\begin{align}\label{gaussmodel}
	\min_{\vz} \frac{1}{m} \sum_{j=1}^{m}\left(|\va_j^*\vz|^2-y_j\right)^2 \hspace{8pt}\textup{or}\hspace{8pt}\min_{\vz} \frac{1}{m} \sum_{j=1}^{m}\left(|\va_j^*\vz|-\sqrt{y_j}\right)^2.
\end{align}
Representative algorithms include Wirtinger Flow (WF) \cite{candes_wf}, Gauss-Newton method\cite{gao2017phaseless}, Reshaped Wirtinger Flow\cite{zhang2016reshaped}, and Truncated Amplitude Flow\cite{wang2018solving}, etc.
To improve computational efficiency, several incremental variants,  such as Incremental Truncated Amplitude Flow\cite{zhang2017phase} and Incremental Reshaped Wirtinger Flow\cite{Zhang2017A}, have been proposed and demonstrated competitive performance in simulations. To simplify the solving process, some revised models are also introduced such as those given in papers \cite{gao2020perturbed} and \cite{luo2020phase}.

In practical scenarios, measurements are often corrupted by Poisson noise, i.e.,
\begin{equation}\label{Poisson_samp}
	y_j \,{\sim} \,\text{Poisson}(|\langle \va_j, \vx\rangle|^2), \quad j=1,2,\ldots, m.
\end{equation}
This is particularly relevant in optical applications, where the intensity correlates with photon counts, and Poisson distributions effectively describe the photon statistics at each pixel. For example, in Fourier ptychographic microscopy (FPM), data sets consist of bright-field and dark-field images with vastly different intensity levels. Bright-field images have higher intensities, leading to higher noise levels compared to dark-field images. In such contexts, the Poisson noise model often outperforms the Gaussian noise model, which neglects the intensity-dependent nature of noise \cite{Yeh2015experimental}.

From the maximum likelihood estimation perspective, the measurement data in (\ref{Poisson_samp}) lead to the following optimization problem:
\begin{align}\label{poissonmodel}
	\min_{\vz} \frac{1}{m} \sum_{j=1}^{m}\left(|\va_j^*\vz|^2-y_j\log(|\va_j^*\vz|^2)\right).
\end{align}
By contrast, assuming the measurements are corrupted by Gaussian noise, the maximum likelihood estimation results in models like (\ref{gaussmodel}). To solve model (\ref{poissonmodel}), the Truncated Wirtinger Flow (TWF) algorithm was proposed in \cite{chen2015solving}. TWF starts with a well-chosen initial point and iteratively updates the estimate using gradient descent, with a truncation procedure to exclude observations that cause large deviations from the mean of $ |\va_j^*\vz|/\|\vz\| $ or $ |\sqrt{y_j}-|\va_j^*\vz|| $. This truncation step plays a critical role in ensuring linear convergence, particularly in the real domain. For incremental adaptations of this algorithm, where only one measurement is processed per iteration, the truncation criteria require modifications to avoid scanning the entire dataset, as discussed in \cite{kolte2016phase}.
%
%
%
\subsection{Goals and Motivations}

The Poisson noise model has gained increasing attention due to its strong relevance to practical applications, particularly in imaging systems where photon statistics play a critical role. Recently, a more generalized observation model has been considered:
\begin{equation}\label{sample2}
	y_j \,{\sim} \,\text{Poisson}\big(|\langle \va_j, \vx\rangle|^2+b_j\big), \quad j=1,2,\ldots, m,
\end{equation}
where $b_j$ represents a known mean background for the $j$-th measurement. This setting, referred to as the Poisson phase retrieval problem, better captures real-world scenarios and has been extensively discussed in \cite{li2021poisson} and \cite{fatima2022pdmm}. 

For measurements following (\ref{sample2}), the maximum likelihood estimation problem takes the form:
\begin{align*}
	\min_{\vz}\frac{1}{m} \sum_{j=1}^{m}\left(|\va_j^*\vz|^2+b_j-y_j\log(|\va_j^*\vz|^2+b_j)\right),
\end{align*}
which extends the simpler model (\ref{poissonmodel}) with an additional background term $b_j$. Based on this model, several algorithms, including WF, ADMM, and MM methods, have been developed \cite{li2021poisson, li2021algorithms}. Notably, \cite{fatima2022pdmm} proposed a primal-dual majorization-minimization (PDMM) algorithm and established its convergence to a stationary point. However, despite these advances, the theoretical guarantees for these algorithms remain limited.

The main goal of this paper is to address this gap by providing a rigorous theoretical analysis for the Wirtinger Flow (WF) algorithm, one of the simplest and most computationally efficient methods for Poisson phase retrieval. Specifically, our contributions are:
\begin{itemize}
	\item We prove that WF for Poisson phase retrieval achieves \textbf{linear convergence to the true signal} with optimal sample complexity under noiseless measurements. This significantly strengthens its theoretical foundation.
	\item We establish the \textbf{stability} of WF for Poisson phase retrieval under bounded noise, demonstrating that the algorithm remains robust even in the presence of random perturbations in the observations.
	\item We introduce an \textbf{incremental version} of WF, designed to enhance computational efficiency by processing one measurement at a time. We rigorously prove that this incremental algorithm converges to the true signal with high probability in the noiseless case.
\end{itemize}

These contributions not only fill the theoretical gaps in existing literature but also highlight the potential of WF and its incremental variant to enable scalable and robust phase retrieval under realistic Poisson noise models.
\subsection{Notations}
Throughout this paper, let $ \vx\in\C^n $ represent the signal we aim to recover, and $ \vz_k $ denote the $ k $-th iteration point of the recovery algorithm. We assume the measurements $ \va_j\sim \mathcal{N}(\0,I_n/2)+i\mathcal{N}(\0,I_n/2), \,\,j=1,2,\ldots, m $ are complex Gaussian random vectors. For a complex number $ x $, $ \Re(x) $ represents its real part. For a set $ I $, $ |I|$ denotes its cardinality. Let $ \|\cdot\| $ denote the Euclidean norm. We use $ C $, $C'$, $c$, $ c' $ or their subscripted forms to denote positive constants, whose values may change from line to line. 
Since $ |\langle\va,\vx \rangle|= |\langle\va,c\vx \rangle|$ for any $ |c|=1 $, the phase retrieval problem cannot distinguish two signals that differ by a unit factor. Therefore, we define the distance between two vectors as:
\[
\text{dist}(\vx,\vz)=\min_{\phi\in[0,2\pi)}\|\vx\cdot e^{i\phi}-\vz\|:=\|\vx\cdot e^{i\phi(\vz)}-\vz\|.
\]
We also define $ \mathcal{S_{\vx}(\rho)} $ as the $ \rho $-neighborhood of $ \vx $:
\[
\mathcal{S_{\vx}(\rho)}:=\{\vz\in\C^n: \text{dist}(\vx,\vz)\leq \rho\|\vx\| \}.
\]
\subsection{Organizations}

The structure of this paper is as follows:  In Section \ref{section2}, we introduce the Wirtinger Flow (WF) algorithm for Poisson phase retrieval and rigorously analyze its convergence properties. Specifically, we establish that WF achieves linear convergence to the exact solution under noiseless measurements and demonstrate robust stability under bounded noise.  
Section \ref{section3} extends WF to an incremental variant, which processes one measurement per iteration to significantly reduce computational cost. We provide a detailed theoretical analysis, proving that the incremental algorithm converges to the true signal with high probability under appropriate conditions.  In Section \ref{section4}, we conduct a series of numerical experiments to validate our theoretical results. These experiments assess the convergence rates under different step sizes, the influence of the background term 
$\vb$, and the comparisons of WF algorithm under different noise settings.
Finally, Section \ref{conclude} summarizes the key contributions of this work and outlines potential directions for future research.  

The appendices provide additional technical details to support the theoretical analysis. Appendix \ref{sec_appA} contains the proofs of two key lemmas regarding the smoothness and curvature conditions, which are essential for ensuring the convergence of the proposed algorithms. Appendix \ref{sec_appB} presents supplementary lemmas and examines the relationships between the parameters used in this paper.  
\section{Wirtinger Flow for Poisson phase retrieval}\label{section2}

We begin by restating the noisy dataset, the corresponding solution model, and the Wirtinger Flow method for Poisson phase retrieval (WF-Poisson). Following the setup in \cite{li2021poisson} and \cite{fatima2022pdmm}, we assume the measurements follow a Poisson distribution:
\begin{equation*}
	y_j \sim \text{Poisson}\big(|\langle \va_j, \vx\rangle|^2 + b_j\big), \quad j = 1, 2, \ldots, m,
\end{equation*}
where $b_j > 0$ is a known mean background term for the $j$-th measurement. For the theoretical analysis, we further assume that $b_j$ is proportional to the true signal intensity, satisfying:
\[
\alpha_1 |\langle \va_j, \vx\rangle|^2 \leq b_j \leq \alpha_2 |\langle \va_j, \vx\rangle|^2,
\]
where $0 < \alpha_1 \leq \alpha_2$ are constants. To ensure the measurements provide sufficient information, we impose the following condition:
\[
|\langle \va_j, \vx\rangle| \geq C_\vx \|\vx\|, \quad j = 1, 2, \ldots, m,
\]
where $C_\vx > 0$ is a constant depending on the signal $\vx$.

The Poisson phase retrieval problem is formulated as an unconstrained minimization problem based on the maximum likelihood estimation (MLE):
\begin{align}\label{ourmodel}
	\min_{\vz} f(\vz) := \frac{1}{m} \sum_{j=1}^{m}\left(|\va_j^* \vz|^2 + b_j - y_j \log(|\va_j^* \vz|^2 + b_j)\right).
\end{align}

\subsection{Wirtinger Flow method for Poisson phase retrieval}

To solve the optimization problem in (\ref{ourmodel}), we apply the Wirtinger Flow method, a gradient descent approach that iteratively updates solution estimates using the Wirtinger derivative of the objective function $ f(\vz)$.  The Wirtinger derivative, which treats $ \vz $ and $ \overline{\vz} $ as independent variables, is commonly used for real-valued functions of complex variables. Given an initial estimate $ \vz_0 $, the iteration updates are defined as:
\begin{equation}\label{iteration}
	\vz_{k+1} = \vz_k-\mu \nabla f(\vz_k),
\end{equation}
where $ \mu $ is the step-size and 
\[
\nabla f(\vz) =\left(\frac{\partial f(\vz,\overline{\vz})}{\partial \vz}\Big|_{\overline{\vz} = \text{constant}}\right)^*= \frac{1}{m}\sum_{j=1}^{m}\left(1-\frac{y_j}{|\va_j^*\vz|^2+b_j}\right)\va_j\va_j^*\vz
\]
is the Wirtinger derivative. 
\begin{remark}
	Obtaining a well-chosen initial estimate is critical for the convergence of WF. While this paper assumes the existence of an initial guess $\vz_0$ satisfying 
	\[
	\dist (\vz_0, \vx) \leq \rho\|\vx\|,
	\]
	for some $\rho > 0$, the specific initialization strategies are beyond the scope of this work. For further details, readers may refer to \cite{candes_wf, wang2018solving, gao2017phaseless, luo2019optimal}.
\end{remark}

\subsection{Convergence result}
We now focus on the convergence analysis of the Wirtinger Flow method for Poisson phase retrieval, starting with the noiseless scenario. Here, we establish that, in the absence of noise, Wirtinger Flow achieves exact recovery with $ m=O(n) $ measurements.
\begin{theorem}[Exact recovery]\label{mainresult1}
	Suppose $ y_j=|\va_j^*\vx|^2+b_j ,\,j=1,2,\ldots, m$ with $\alpha_1|\langle \va_j, \vx\rangle|^2 \leq b_j\leq \alpha_2|\langle \va_j, \vx\rangle|^2 $. Here $ \alpha_1 $ and $ \alpha_2 $ are positive constants. Starting from an initial point $ \vz_0\in\mathcal{S}_\vx(\rho) $, we solve (\ref{ourmodel}) iteratively by
	\[
	\vz_{k+1} = \vz_k-\mu \nabla f(\vz_k),
	\]
	with step size $ \mu<2l_{cur}/u^2_{smo}  $. Here $l_{cur} $ and $ u_{smo} $ are positive constants defined in Lemmas \ref{smoothness} and \ref{curvature condition}, respectively. 
	Then for a sufficiently large constant $ C $, 
	when the measurements $ m\geq Cn $, with probability at least $ 1-(k+1)\exp(-cn) $, we have 
	\[
	\textup{dist}^2(\vz_{k+1}, \vx)\leq (1-t)^{k+1}\cdot\textup{dist}^2(\vz_0,\vx)
	\]
	with $ t=\mu(2\cdot l_{cur}-\mu \cdot u^2_{smo})>0 $.
\end{theorem}
This theorem guarantees the linear convergence of the algorithm to an exact solution. As with most non-convex algorithms, this result relies on both a local smoothness condition (Lemma \ref{smoothness}) and a local curvature condition (Lemma \ref{curvature condition}).

\begin{proof}
	Under the conditions given, the smoothness condition (Lemma \ref{smoothness}) provides an upper bound on $ \|\nabla f(\vz_k)\| $:
	\[
	\|\nabla f(\vz_k)\|\leq u_{smo}\cdot\dist(\vx,\vz_k)
	\]
	and the curvature condition (Lemma \ref{curvature condition}) provides a lower bound on $ \Re\big( \langle\nabla f(\vz_k), \vz_k-\vx e^{i\phi(\vz_k)}\rangle\big) $:
	\[
	\Re\big( \langle\nabla f(\vz_k), \vz_k-\vx e^{i\phi(\vz_k)}\rangle\big)\geq l_{cur}\cdot\dist^2(\vx,\vz_k).
	\]
	Here $ u_{smo} $ and $ l_{cur} $ are positive constants depending on $ \alpha_1 $, $ \alpha_2 $ and $ \rho $. Therefore, when $ m\geq Cn $ and $  \vz_{k+1}\in \mathcal{S}_\vx(\rho)$, with probability at least $ 1-\exp(-cn) $, we have 
	\begin{equation}\label{iter_down}
		\begin{aligned}
			\dist^2(\vz_{k+1},\vx)&=\|\vz_{k+1}-e^{i\phi(\vz_{k+1})}\vx\|^2\leq \|\vz_{k+1}-e^{i\phi(\vz_{k})}\vx\|^2=\|\vz_k-e^{i\phi(\vz_{k})}\vx-\mu\nabla f(\vz_k)\|^2\\
			&\leq \dist^2(\vz_{k},\vx)+\mu^2\|\nabla f(\vz_k)\|^2-2\mu\Re\big( \langle\nabla f(\vz_k), \vz_k-\vx e^{i\phi(\vz_k)}\rangle\big)\\
			&\leq ( 1+\mu^2\cdot u^2_{smo}-2\mu \cdot l_{cur})\cdot\dist^2(\vz_{k},\vx)\\
			&=\big(1-\mu(2\cdot l_{cur}-\mu \cdot u^2_{smo})\big)\cdot\dist^2(\vz_{k},\vx).
		\end{aligned}
	\end{equation}	
	We can then obtain a refined iteration point by simply choosing a step size $ \mu $ satisfying $ 2\cdot l_{cur}-\mu \cdot u^2_{smo} >0$, which gives $ \mu<2\cdot l_{cur}/u^2_{smo} $. 
	For convenience, we set $ t:=\mu(2\cdot l_{cur}-\mu \cdot u^2_{smo})>0 $. Then from (\ref{iter_down}), we have
	\begin{equation}\label{iter_dowm1}
		\dist^2(\vz_{k+1},\vx)\leq (1-t)\dist^2(\vz_{k},\vx)\leq (1-t)^2\dist^2(\vz_{k-1},\vx)\leq(1-t)^{k+1}\dist^2(\vz_{0},\vx),
	\end{equation}
	which demonstrates the linear convergence of the algorithm.
\end{proof}

\begin{remark}\label{set_stepsize}
	The constants $ u_{smo} $ and $ l_{cur} $ depend on the values of $ \alpha_1 $, $ \alpha_2 $ and $ \rho $. The specific relationships are provided in   Lemmas \ref{smoothness} and \ref{curvature condition}. In particular, if we set $ \rho = 1/15 $ and $ 0.8\leq\alpha_1\leq\alpha_2\leq1.2 $,  we  will have $l_{cur} =0.0126 $ and $ u_{smo}=1.58 $.	 
\end{remark}

Next, we analyze the case with noise. The following theorem shows that WF attains $ \varepsilon $ accuracy in a relative sense within $ \mathcal{O}(\log(1/\varepsilon)) $ iterations, matching the result of the Truncated Wirtinger Flow method.
\begin{theorem}[Stability]
	Consider the noisy case where $ y_j =|\langle \va_j, \vx\rangle|^2+b_j+\eta_j,\,j=1,2,\ldots,m $ with $ \bm{\eta}=(\eta_1,\ldots,\eta_m) $ representing the noise. As before, assume $\alpha_1|\langle \va_j, \vx\rangle|^2 \leq b_j\leq \alpha_2|\langle \va_j, \vx\rangle|^2 $, where $ \alpha_1 $ and $ \alpha_2 $ are positive constants.
	Suppose the noise is bounded such that $ \frac{\|\bm\eta\|}{\sqrt{m}}\leq c\|\vx\| $ for some positive constant $ c $. Then when $ m\geq Cn $ with $ C $ sufficiently large, with probability greater than $ 1-c_1\exp(-c_2m) $, starting from any $ \vz_0\in\mathcal{S_{\vx}(\rho)} $ the iteration (\ref{iteration}) yields 
	\[
	\dist(\vz_{k+1},\vx)\leq c_3\frac{\|\bm \eta \|}{\sqrt{m}\|\vx\|} + (1-t_1)^{\frac{k+1}{2}}\|\vx\|.
	\]
	Here $ c_3 $, $ t_1 $ are positive constants.
\end{theorem}
\begin{proof}
	Under noisy measurements, we first decompose the gradient at $ \vz $ as:
	\begin{align*}
		\nabla f(\vz)
		&=\frac{1}{m}\sum_{j=1}^{m}\left(1-\frac{|\va_j^*\vx|^2+b_j+\eta_j}{|\va_j^*\vz|^2+b_j}\right)\va_j\va_j^*\vz\\
		&=\frac{1}{m}\sum_{j=1}^{m}\left(1-\frac{|\va_j^*\vx|^2+b_j}{|\va_j^*\vz|^2+b_j}\right)\va_j\va_j^*\vz-\frac{1}{m}\sum_{j=1}^{m}\frac{\eta_j}{|\va_j^*\vz|^2+b_j}\va_j\va_j^*\vz\\
		&:=\nabla f_{exact}(\vz)-\nabla f_{noise}(\vz).
	\end{align*}
	Define $ \bm w = (w_1, w_2,\ldots,w_m)\zz $ with $ w_j = \frac{\eta_j}{|\va_j^*\vz|^2+b_j}\va_j^*\vz $, allowing us to rewrite the noise-induced gradient term as $\nabla f_{noise}(\vz) =\frac{1}{m} A\bm w  $. Observing that 
	$
	|w_j|\leq\frac{|\eta_j||\va_j^*\vz|}{2\sqrt{\alpha_1}|\va_j^*\vz||\va_j^*\vx|}\leq \frac{|\eta_j|}{2\sqrt{\alpha_1}C_\vx\|\vx\|}
	$,  we obtain
	$
	\|\bm w\|^2\leq \frac{\|\bm \eta\|^2}{4\alpha_1C^2_\vx\|\vx\|^2}.
	$
	By Lemma \ref{sub_gaussian_concentration}, for any $ \delta>0 $, we have
	\begin{equation}\label{fnoisenorm}
		\|\nabla f_{noise}(\vz)\| =\left\|\frac{1}{m} A\bm w \right\|\leq \left\|\frac{1}{\sqrt{m}}A\right\|\left\|\frac{1}{\sqrt{m}}\bm w\right\|\leq (1+\delta)\frac{\|\bm\eta\|/\sqrt{m}}{2\sqrt{\alpha_1}C_\vx\|\vx\|}
	\end{equation}
	with probability greater than $ 1-\exp(-c'n) $ if $ m/n $ sufficiently large.
	
	From Theorem \ref{mainresult1}, we know that when the gradient at $ \vz_{k} $ satisfies the local smoothness and curvature conditions, the distance between $ \vz_{k+1}=\vz_k-\mu\nabla f(\vz_k) $ and the exact signal $ \vx $ can be reduced with an appropriate step size.
	However, these conditions may not be satisfied due to the presence of noise. We now examine two cases based on whether these conditions are satisfied:
	\begin{itemize}
		\item Case 1: Suppose 
		\[
		c_1\frac{\|\bm \eta\|/\sqrt{m}}{\|\vx\|}\leq \|\vh\|\leq c_2
		\]  
		where $ \vh = e^{-i\phi(\vz)}\vz-\vx $ and $ c_1 $ is a sufficiently large constant depending on $ l_{cur} $ (requirements specified later).
		In this case, the noise-induced gradient term satisfies
		\[
		\|\nabla f_{noise}(\vz)\|\leq (1+\delta)\frac{\|\bm\eta\|/\sqrt{m}}{2\sqrt{\alpha_1}C_\vx\|\vx\|}\leq \frac{(1+\delta)}{2\sqrt{\alpha_1}C_\vx\cdot c_1}\|\vh\|:=\epsilon_{c_1}\|\vh\|,
		\]
		where $ \epsilon_{c_1}=\frac{(1+\delta)}{2\sqrt{\alpha_1}C_\vx\cdot c_1} $.
		Thus, the smoothness condition holds:
		\[
		\|\nabla f(\vz)\|\leq \|\nabla f_{exact}(\vz)\|+\|\nabla f_{noise}(\vz)\|    \leq (u_{smo}+\epsilon_{c_1})\|\vh\|.
		\]
		For the curvature condition, we obtain
		\[
		|\Re\left( \langle\nabla f_{noise}(\vz), \vh\rangle\right)|\leq \|\nabla f_{noise}(\vz)\|\cdot\|\vh\|\leq \epsilon_{c_1}\|\vh\|^2,
		\]
		so
		\begin{align*}
			|\Re\left( \langle\nabla f(\vz), \vh\rangle\right)|&=|\Re\left( \langle\nabla f_{exact}(\vz), \vh\rangle\right)-\Re\left( \langle\nabla f_{noise}(\vz), \vh\rangle\right)|\\
			&\geq l_{cur}\|\vh\|^2-\|\nabla f_{noise}(\vz)\|\cdot\|\vh\|\\
			&\geq (l_{cur}-\epsilon_{c_1})\|\vh\|^2.
		\end{align*}
		Thus if  $c_1$ is chosen so that $\epsilon_{c_1}=\frac{(1+\delta)}{2\sqrt{\alpha_1}C_\vx\cdot c_1}<l_{cur}$, i.e., $ c_1>\frac{(1+\delta)}{2\sqrt{\alpha_1}C_\vx\cdot l_{cur}} $, the curvature condition holds. 
		Therefore, by Theorem \ref{mainresult1}, there exists  $ t_1>0 $ such that 
		\[
		\dist(\vz_{k+1},\vx)\leq (1-t_1)^{\frac{k+1}{2}}\|\vx\|.
		\]
		\item Case 2: Suppose 
		\[
		\|\vh\|\leq c_1\frac{\|\bm \eta\|/\sqrt{m}}{\|\vx\|}.
		\]
		Here the iteration may not reduce the distance between the iterate and the ground truth. However, each step only changes the distance by at most $ O(\|\bm \eta\|/\sqrt{m})$,  so the error cannot increase by more than a constant multiple of $ \|\bm \eta\|/\sqrt{m} $.
		When $ \|\bm \eta\|/\sqrt{m} $ is small enough to ensure $c_1\frac{\|\bm \eta\|/\sqrt{m}}{\|\vx\|}\leq c_2  $, we have 
		\[
		\dist(\vz_{k+1},\vx)\leq c_3 \frac{\|\bm \eta\|}{\sqrt{m}\|\vx\|},
		\]
		where $ c_3 $ is a positive constant. As iterations progress, the sequence will eventually reach Case 1.
	\end{itemize}
	Thus, we have established the stability of the Wirtinger Flow algorithm under bounded noise.
\end{proof}
\section{Incremental Wirtinger Flow for Poisson phase retrieval}\label{section3}
In this section, we further develop an incremental variant of the Wirtinger Flow method, hereafter referred to as Incremental Wirtinger Flow (IWF). In contrast to traditional WF, which calculates the gradient in each iteration as the average over the entire dataset:\newline
$ \{(\va_1,y_1,b_1),(\va_2,y_2,b_2),\ldots,(\va_m,y_m,b_m)\} $, IWF uses only a single randomly selected observation $ \{(\va_{i_k},y_{i_k},b_{i_k})\} $ per iteration. This modification reduces the computational cost per iteration by a factor of  $m$, making IWF advantageous for large-scale problems.

As previously discussed, although various incremental algorithms have been proposed, our IWF distinguishes itself by its implementation simplicity. Unlike other approaches, IWF requires no additional truncation condition checks to ensure convergence, which streamlines the update steps and reduces computational complexity. This advantage makes IWF particularly attractive for large-scale applications, where efficient and straightforward algorithms are essential. 

To perform IWF, we select an index $ i_s $ uniformly at random from $ \{1,2,\ldots, m\} $, and then use the observation $ \{(\va_{i_s},y_{i_s},b_{i_s})\} $ for the $ s $-th iteration:
\begin{align}\label{inc_iter}
	\vz^{(s+1)}&=\vz^{(s)} -\mu\nabla f_{i_s}(\vz^{(s)})
\end{align}
with 
\[
\nabla f_{i_s}(\vz^{(s)})=\left(1-\frac{y_{i_s}}{|\va_{i_s}^*\vz^{(s)}|^2+b_{i_s}}\right)\va_{i_s}\va_{i_s}^*\vz^{(s)}.
\]

Note that there are two sources of randomness here: one in the generation of the measurements $ \va_j $ and the other in the selection of observations at each algorithm iteration. Here we focus only on the first source of randomness. For the second source, methods from \cite{jeong2017convergence, tan2019phase} can be applied to establish a convergence basin, with hitting time used to record when iteration points exit this basin.	
Using a similar analytical approach, we demonstrate that the IWF method can linearly converge to the ground truth in expectation with high probability.
\begin{theorem}
	Under the same settings as in Theorem \ref{mainresult1}, let the 
	initial point $ \vz^{(1)}$ satisfy $ \vz^{(1)}\in\mathcal{S}_\vx(\rho) $ with $ \rho<1 $. Solving the problem by executing iteration step (\ref{inc_iter}), there exist constants $ C,\,C_1,\,c_1,\,c_2$ such that when $ \mu<\frac{c_1}{n} $ and $ m\geq Cn $, we have, with probability greater than $ 1-C_1m\exp(-c_2n) $,
	\[
	\E_{\mathcal{I}_s}\left[\dist^2(\vz^{(s+1)},\vx) \right]\leq \left(1-\frac{t}{n}\right)^{s}\|\vx\|,
	\]
	where $ \E_{\mathcal{I}_s} [\cdot] $ denotes the expectation with respect to the algorithm randomness $ \mathcal{I}_s=\{i_1,i_2,\ldots,i_s\} $, conditioned on the high probability event of random measurements $ \{\va_j\}_{j=1}^{m} $.
\end{theorem}
\begin{proof}
	To prove the theorem, it suffices to demonstrate that
	\[
	\E_{i_s}\left[\dist^2(\vz^{(s+1)},\vx)\right]\leq\left(1-\frac{t}{n}\right)\dist^2(\vz^{(s)},\vx)   
	\]
	holds with high probability for all $ \vz^{(s)} $ such that $ \dist(\vz^{(s)},\vx) \leq \rho \|\vx\|$.
	
	We define $ \vh= e^{-i\phi(\vz^{(s)})}\vz^{(s)}-\vx $.  
	Following a similar approach as in the proof of Theorem \ref{mainresult1}, we have
	\begin{equation}\label{inc_con}
		\begin{aligned}
			\E_{i_s}\left[\dist^2(\vz^{(s+1)},\vx)\right]&\leq\E_{i_s}\left[\|\vz^{(s+1)}-e^{i\phi(\vz^{(s)})}\vx\|^2\right]\\
			&=\E_{i_s}\left[\|\vz^{(s)}-e^{i\phi(\vz^{(s)})}\vx-\mu\nabla f_{i_s}(\vz^{(s)})\|^2\right]\\
			&=\dist^2(\vz^{(s)},\vx)+\frac{\mu^2}{m}\sum_{j=1}^{m}\bigg(1-\frac{y_j}{|\va_j^*\vz^{(s)}|^2+b_j}\bigg)^2\|\va_j\|^2|\va_{j}^*\vz^{(s)}|^2\\
			&\quad -\frac{2\mu}{m}\sum_{j=1}^{m}\left(1-\frac{y_j}{|\va_j^*\vz^{(s)}|^2 +b_j}\right)\Re\big((\va_j^*\vz^{(s)})\,e^{-i\phi(\vz^{(s)})}(\vh^*\va_j)\big)
		\end{aligned}
	\end{equation}  
	Next, we define the event $ E:=\{\max_{1\leq j\leq m} \|\va_j\|^2\leq 6n\} $, which holds with probability $1-m\exp(-1.5n) $. Under this event and using Lemma \ref{smoothness} and Lemma \ref{curvature condition}, we obtain
	\begin{align}\label{inc_con1}
		\frac{1}{m}\sum_{j=1}^{m}\bigg(1-\frac{y_j}{|\va_j^*\vz^{(s)}|^2+b_j}\bigg)^2\|\va_j\|^2|\va_{j}^*\vz^{(s)}|^2\leq c_{11}n\|\vh\|^2
	\end{align}
	and
	\begin{align}\label{inc_con2}
		\frac{1}{m}\sum_{j=1}^{m}\left(1-\frac{y_j}{|\va_j^*\vz^{(s)}|^2 +b_j}\right)\Re\big((\va_j^*\vz^{(s)})\,e^{-i\phi(\vz^{(s)})}(\vh^*\va_j)\big)\geq c_{12}\|\vh\|^2,
	\end{align}
	where $ c_{11} $ and $ c_{12} $	are positive constants. By setting $ \mu < \frac{2c_{12}}{nc_{11}}$ and substituting inequalities (\ref{inc_con1}) and (\ref{inc_con2}) into (\ref{inc_con}), we derive
	\begin{align*} 
		\E_{i_s}\left[\dist^2(\vz^{(s+1)},\vx)\right]&\leq
		(1+\mu^2 c_{11}n-2\mu c_{12})\cdot\dist^2(\vz^{(s)},\vx)<\left(1-\frac{t}{n}\right)\cdot\dist^2(\vz^{(s)},\vx),
	\end{align*}
	where $ t>0 $. This completes the proof. 
\end{proof}
\section{Numerical experiments}\label{section4}  
In this section, we conduct a series of numerical experiments to validate our theoretical findings. The first experiment focuses on determining an appropriate step size and the optimal number of measurements for the WF-Poisson algorithm. The second experiment investigates the impact of the choice of $\vb$ on the convergence behavior of the WF-Poisson algorithm. Finally, the third experiment compares the recovery performance of the WF algorithms under both Poisson noise (WF-Poisson) and Gaussian noise (WF-Gaussian), highlighting the advantages of the WF-Poisson model in handling Poisson noise.

\subsection{Step size selection}    
To optimize the performance of the WF-Poisson algorithm, we first aim to identify a suitable step size $\mu$ that ensures efficient convergence. Below, we outline the experimental setup and step size strategies under consideration.

The signal $\vx\in\C^n$ is generated as a Gaussian random vector with $n = 100$.
The measurement vectors $\va_j\in\C^n,\,j=1,2,\ldots,m $ are 
sampled independently from a complex Gaussian distribution. Each $\va_j$ is scaled by a constant to ensure that the average value of $ |\va_j^*\vx|^2 $ is 2. 
For each $j$, we define $b_j = |\va_j^*\vx|^2$, corresponding to the case where $\alpha_1=\alpha_2=1$.
The Poisson noise is modeled as $\eta_j = \eta\cdot\text{Poisson}(|\va_j^*\vx|^2+b_j) $, where $\eta $ controls the noise level. The observed measurements are then expressed as: $y_j=|\va_j^*\vx|^2+b_j+\eta_j $, $j=1,2,\ldots,m$. For the iterative update rule:
\[
\vz_{k+1} = \vz_k - \mu \nabla f(\vz_k),
\]
we evaluate three different strategies for selecting the step size $\mu$.
\par
\textbf{1. Heuristic step size:}

Following the heuristic step size strategy proposed in \cite{candes_wf}, the step size $\mu$ is defined as:
\[
\mu = \min\big(1-\exp(-t/330),\, 0.2\big),
\]
where $t$ is the iteration index. This approach starts with a small step size in the early iterations and gradually increases as the number of iterations grows, promoting faster convergence in later stages.
\par
\textbf{2. Constant step size:}

Theoretically, we have proven that when $\rho \leq 1/15$ and $ 0.8\leq \alpha_1\leq \alpha_2 \leq 1.2$, the step size $\mu \leq 0.01$ (Remark \ref{set_stepsize}) guarantees linear convergence. However, in practice, larger step sizes can also achieve linear convergence. Here we set $\mu = 0.2$ as a constant step size for all iterations, which aligns with the maximum step size used in the heuristic strategy.     
\par
\textbf{3. Step size based on observed Fisher information:}

This method utilizes the observed Fisher information, as introduced in \cite{li2021algorithms}. The step size $\mu$ is calculated as:
\[
\mu = \frac{\|\nabla f(\vz_k)\|^2}{(A \nabla f(\vz_k))^* D (A \nabla f(\vz_k))},
\]
where $D = \text{Diag}\big(|A\vz_k|^2./(|A\vz_k|^2+b)\big) $. This adaptive step size formulation ensures that the update direction is scaled appropriately based on the curvature of the loss function.

To evaluate the effectiveness of the three step size strategies, we conducted two experiments: success rate analysis and iteration convergence analysis. The success rate experiment assesses the recovery performance of the step size strategies and determines the minimum number of measurements required for stable and accurate signal recovery. The iteration convergence experiment then compares the convergence speed of the strategies under fixed measurement conditions.

In the success rate experiment, the number of measurements  $ m $ is varied from $3n$ to $5n$ in increments of $0.2n$, allowing a systematic investigation of how the  quantity of measurementaffects recovery performance. A recovery trial is deemed successful if the normalized root mean square error (NRMSE) of the reconstructed signal is less than $ 0.5*10^{-3}$ for $\eta=10^{-3}$ or $0.5$ for $\eta=0.1$ within 500 iterations. For each value of $m$, 100 independent trials are conducted, with randomized signals and measurement matrices. The success rate is calculated as the proportion of successful trials. As shown in Figure \ref{success_m_n}, a higher measurement-to-signal ratio ($ m/n$) consistently improves recovery stability, with all step size strategies performing similarly. Based on these results, we set $ m=5n$ for subsequent experiments to ensure stable and reliable convergence.

In the iteration convergence analysis, we fix the number of measurements at  $m=5n$ and compare the iteration descent of NRMSE under the same noise levels ($\eta=10^{-3}$ and $\eta=0.1$). The results, averaged over $50$ independent trials, are shown in Figure \ref{nrmse_iterate}. Although the success rate experiment indicates that all step size strategies achieve comparable recovery performance, the convergence analysis reveals that the constant step size achieves slightly faster convergence compared to the other strategies. Based on these findings, we adopt the constant step size for all subsequent experiments to ensure efficient and reliable algorithm implementation.

\begin{figure}[htbp!]   	
	\begin{center}
		\subfigure[]{
			\includegraphics[width=0.45\textwidth]{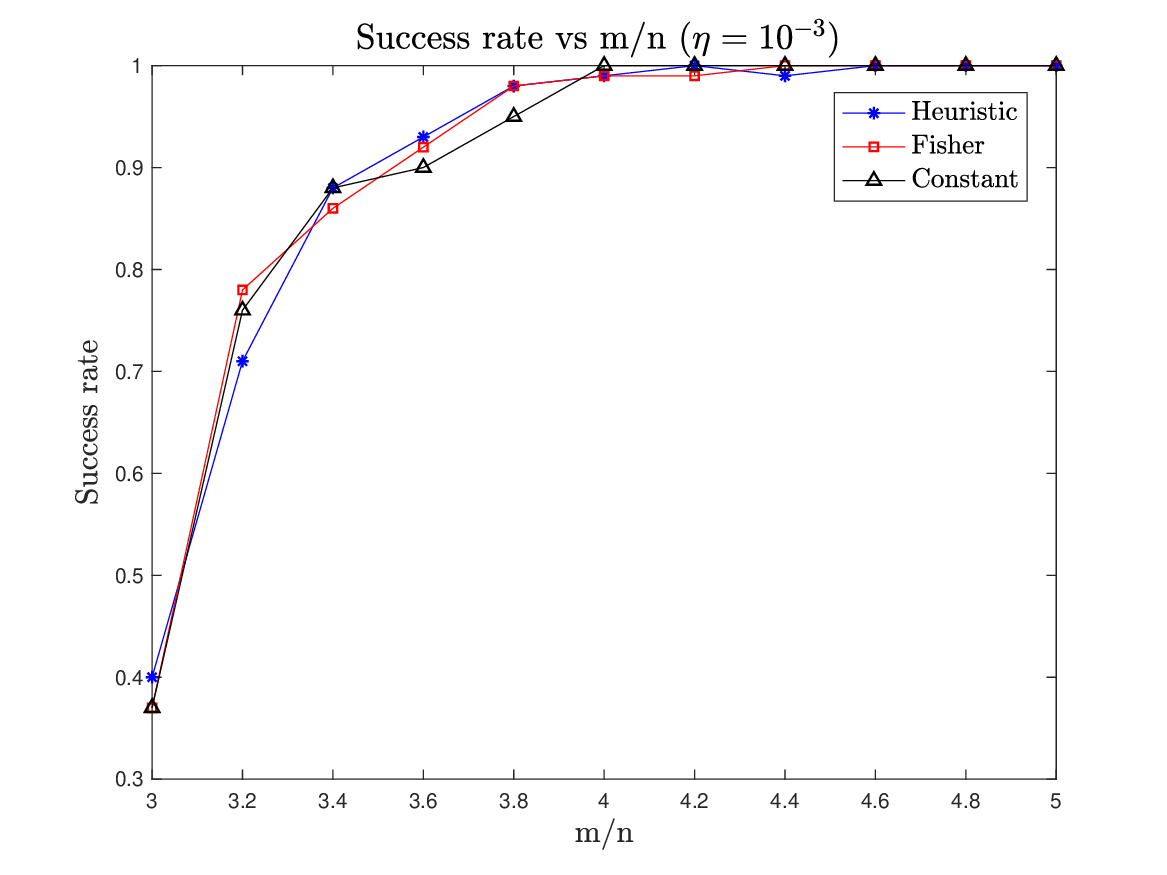}}
		\subfigure[]{
			\includegraphics[width=0.45\textwidth]{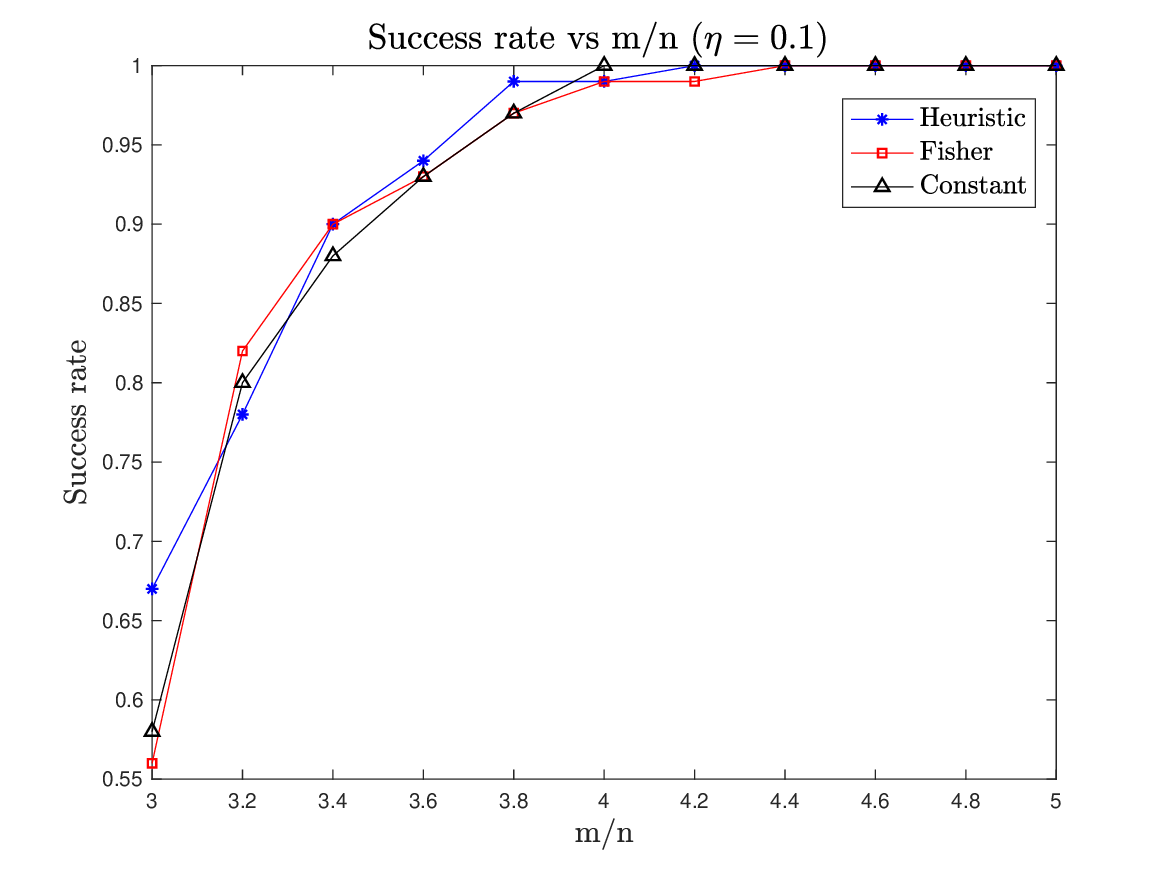}}
		\caption{Success rate: The success rate is calculated over 100 trials for varying measurement numbers, with $ n = 100 $ and $ m/n\in[3:0.2:5] $, under noise levels (a) $\eta = 10^{-3}$ and (b) $\eta = 0.1$.
		}\label{success_m_n}
	\end{center}
\end{figure}

\begin{figure}[htbp]
	\begin{center}
		\subfigure[]{
			\includegraphics[width=0.45\textwidth]{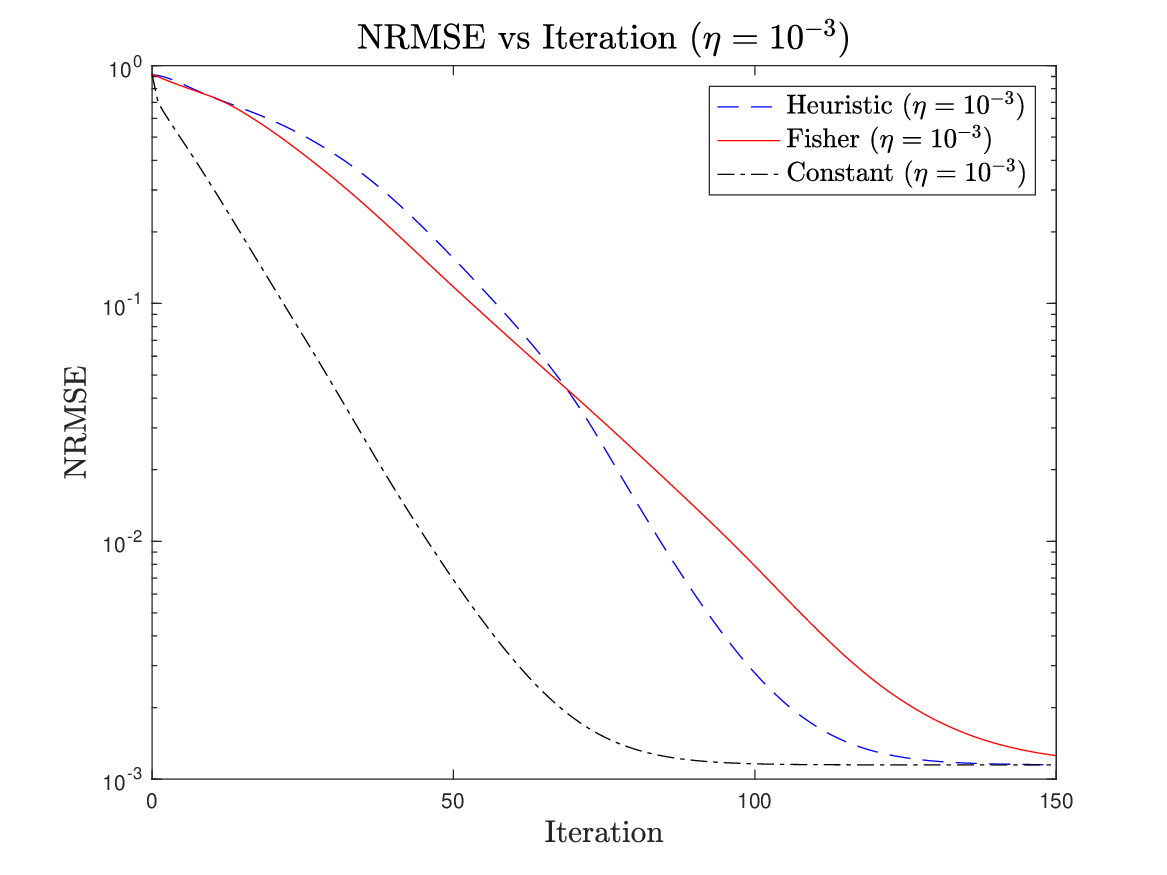}}
		\subfigure[]{
			\includegraphics[width=0.45\textwidth]{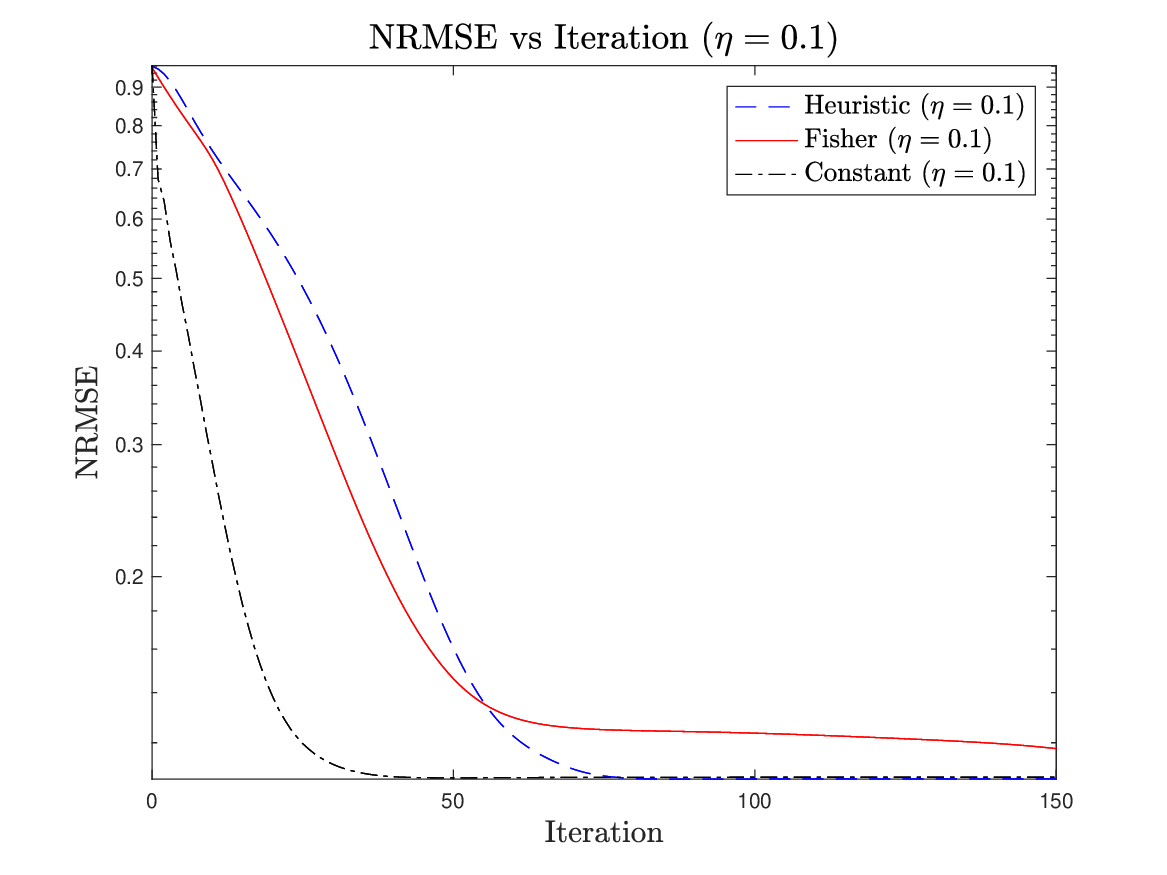}}
		\caption{Convergence: NRMSE per iteration, with $ n = 100 $, $ m=5n $ under  noise levels (a) $\eta = 10^{-3}$ and (b) $\eta = 0.1$.
		}\label{nrmse_iterate}
	\end{center}
\end{figure}   
\subsection{The influence of $\vb$.}    

To investigate the impact of the variable $b$ on the convergence behavior of the algorithm, we set  $\alpha_1=[0.01,0.05,0.1,0.5,1]$ and define  $\alpha_2 = 1/\alpha_1$. The variable $b$ is constrained to satisfy $\alpha_1|\langle \va_j, \vx\rangle|^2 \leq b_j\leq \alpha_2|\langle \va_j, \vx\rangle|^2 $. To achieve this, we construct $b = s.*|A\vx|^2$, where $s$ is a random vector distributed as $s\sim  \exp\big(\text{rand}(m,1)*(\log(\alpha_2)-\log(\alpha_1))+\log(\alpha_1)\big)$. This construction ensures that $b$ is distributed within the interval $[\alpha_1,\alpha_2]\cdot|A\vx|^2$, while maintaining $\mathbb{E}(\log(\vb)) = \mathbf{0}$.
%
We conduct experiments under two noise levels, $\eta = 10^{-3}$ and $\eta = 0.1$. The results, presented in Figure \ref{b_inf}, demonstrate that the algorithm consistently converges under all tested settings. Moreover, the convergence properties improve significantly when both $\alpha_1$ and $\alpha_2$ are closer to 1, indicating that a tighter concentration of $\vb$ around $|A\vx|^2$ enhances the efficiency of the algorithm.
\begin{figure}[htbp]
	\begin{center}
		\subfigure[]{
			\includegraphics[width=0.45\textwidth]{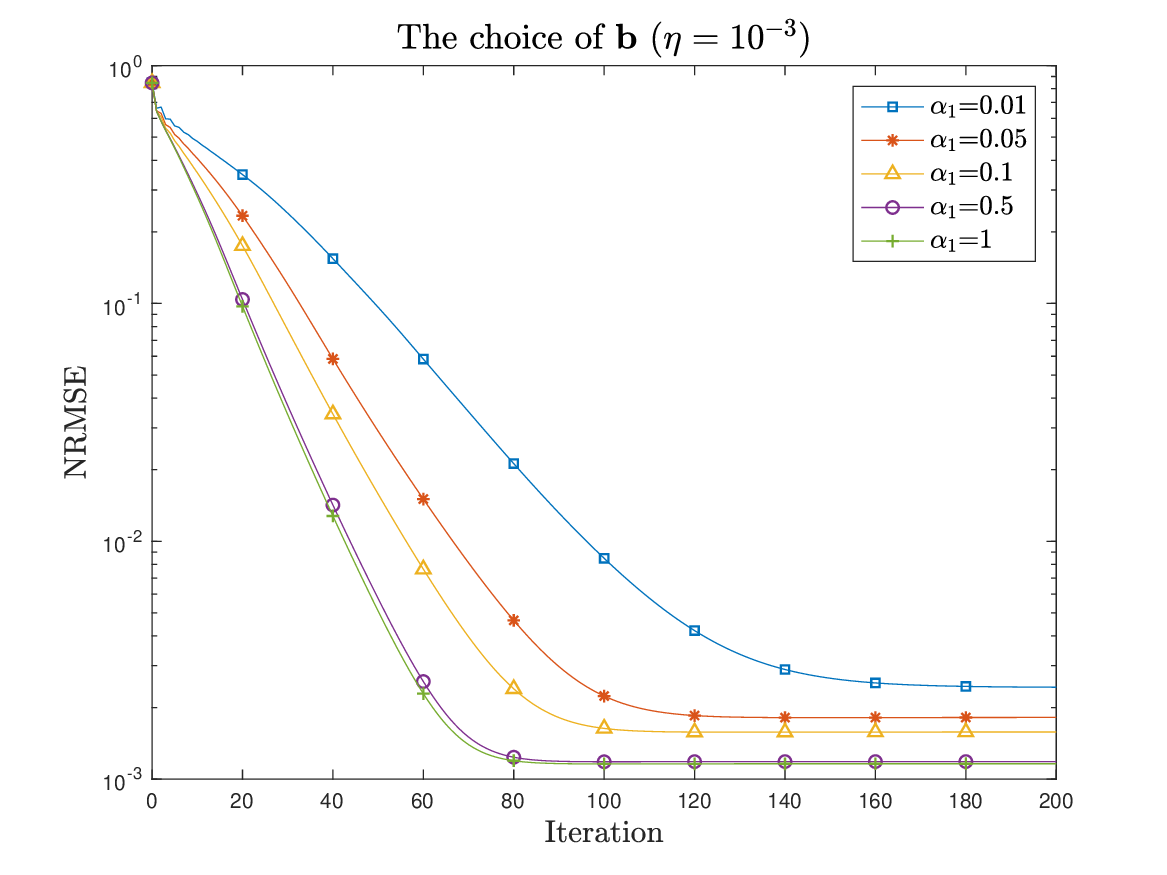}}
		\subfigure[]{
			\includegraphics[width=0.45\textwidth]{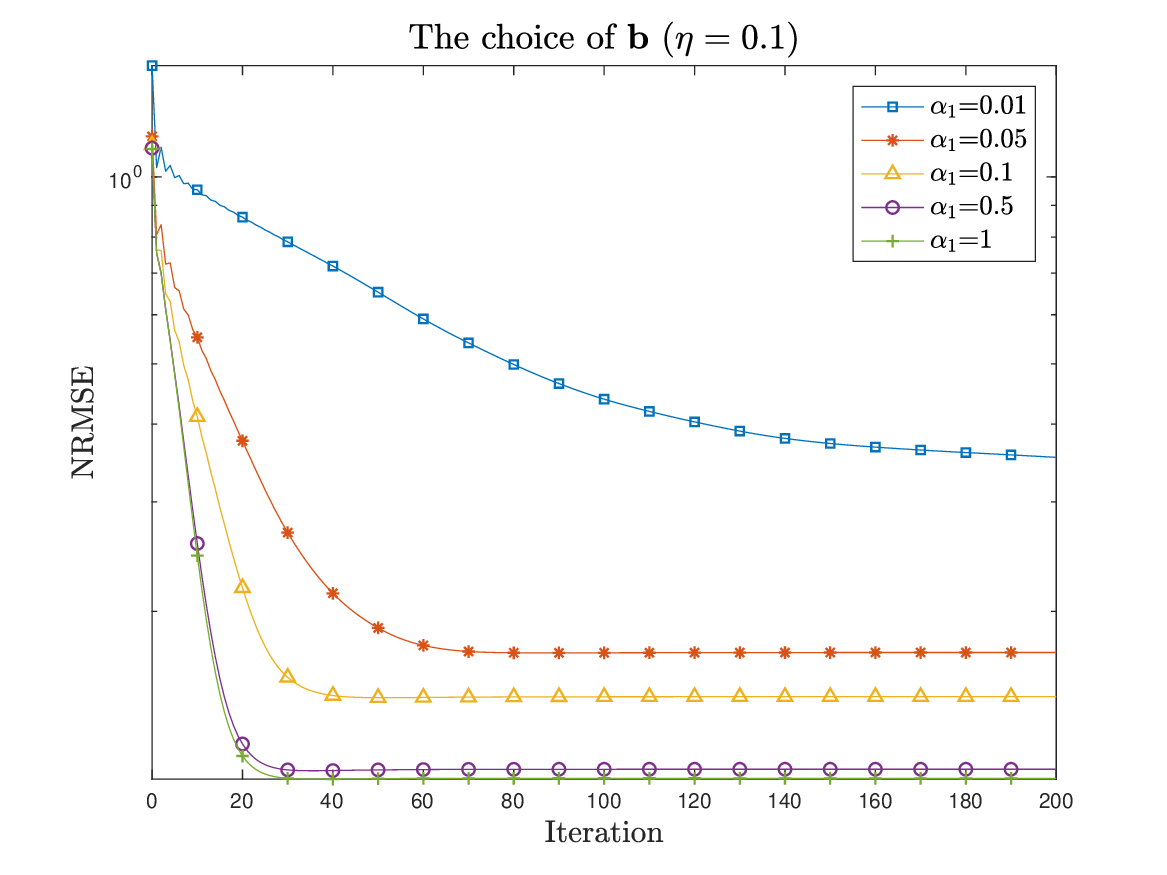}}
		\caption{Influence of $b$: record the NRMSE of each iteration step with different $\alpha_1$ with (a)  $\eta = 10^{-3}$ and (b) $\eta = 0.1$. Here $ n = 100 $, $ m=5n $, $\alpha_2 = 1/\alpha_1$.}\label{b_inf}
	\end{center}
\end{figure}   
\subsection{Compare with Gaussian model.}  

In this experiment, we compare the performance of the WF-Poisson algorithm and the WF-Gaussian algorithm under two different noise types: Poisson noise and Gaussian noise. The WF-Gaussian algorithm is implemented according to the method described in \cite{candes_wf}. To evaluate their performance, we compute the NRMSE after 500 iterations for both algorithms while varying the measurement ratio $m/n=[3:0.2:5]$. The experiments are conducted under two noise levels, $\eta = 10^{-3}$ and $\eta = 0.1$. For each configuration, the NRMSE is averaged over $50$ independent randomized trials to ensure robust comparison.

The results, shown in Figures \ref{compare}, present (a) and (b) for Poisson noise and (c) and (d) for Gaussian noise. The findings demonstrate that the WF-Poisson algorithm achieves superior recovery accuracy under Poisson noise, whereas the WF-Gaussian algorithm performs better under Gaussian noise. These findings emphasize the critical importance of selecting a model that aligns with the underlying noise distribution.

\begin{figure}[htbp]
	\centering
	\subfigure[Poisson noise ($\eta=10^{-3}$)]{
		\includegraphics[width=0.45\textwidth]{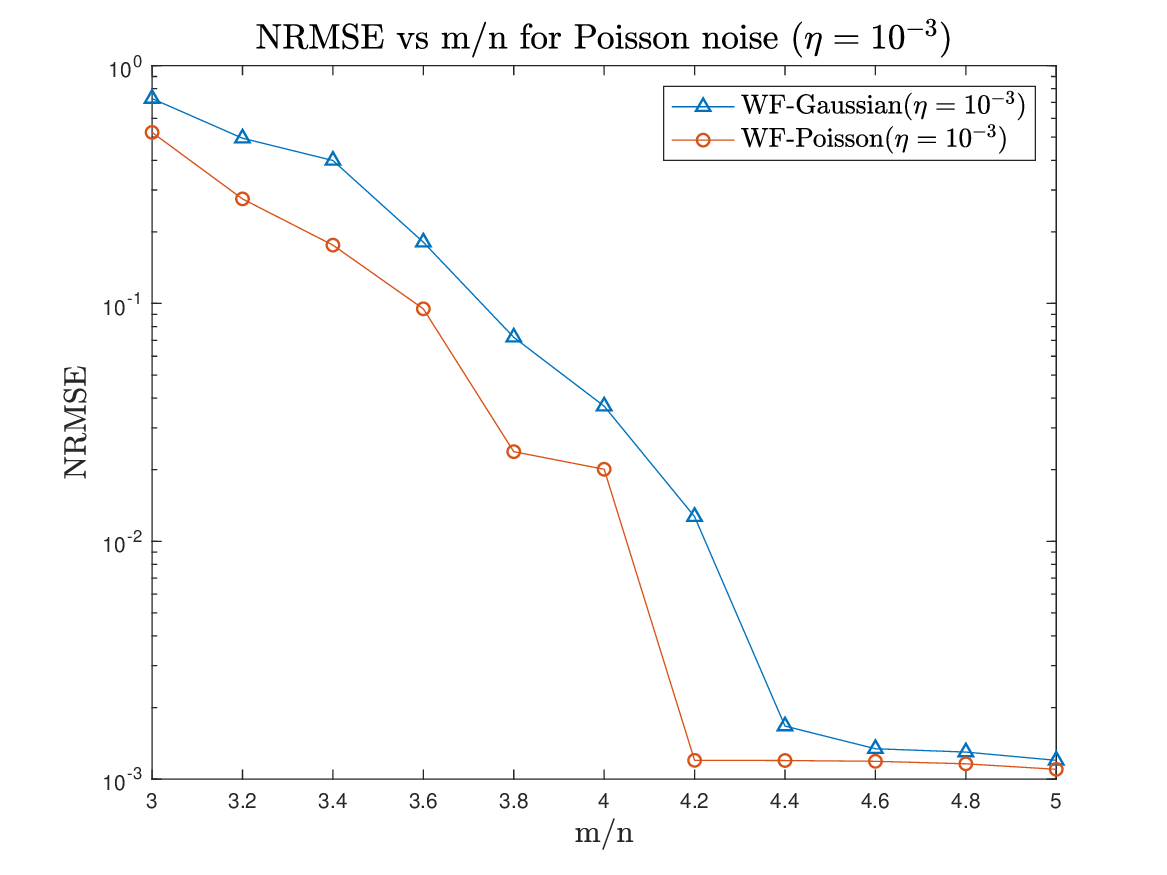}}
	\quad
	\subfigure[Poisson noise ($\eta=0.1$)]{
		\includegraphics[width=0.45\textwidth]{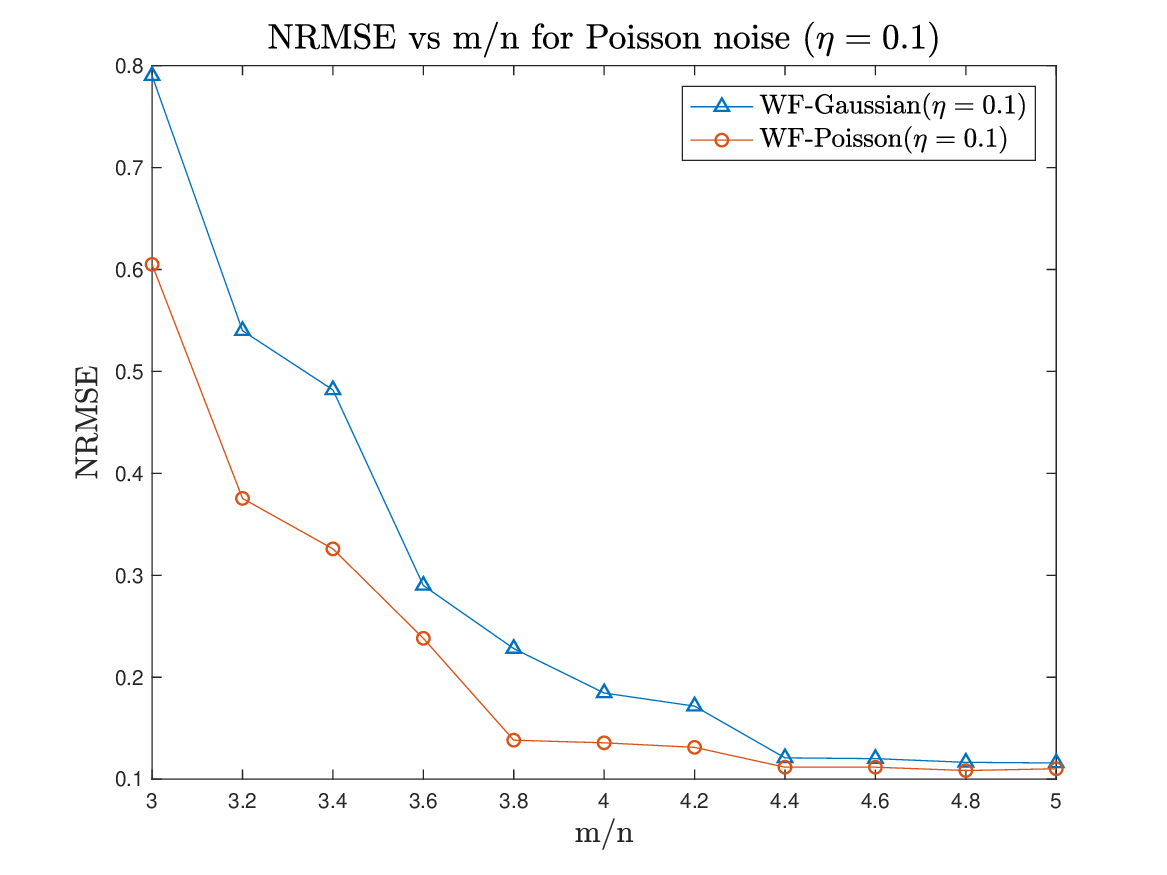}}
	\\
	\subfigure[Gaussian noise ($\eta=10^{-3}$)]{
		\includegraphics[width=0.45\textwidth]{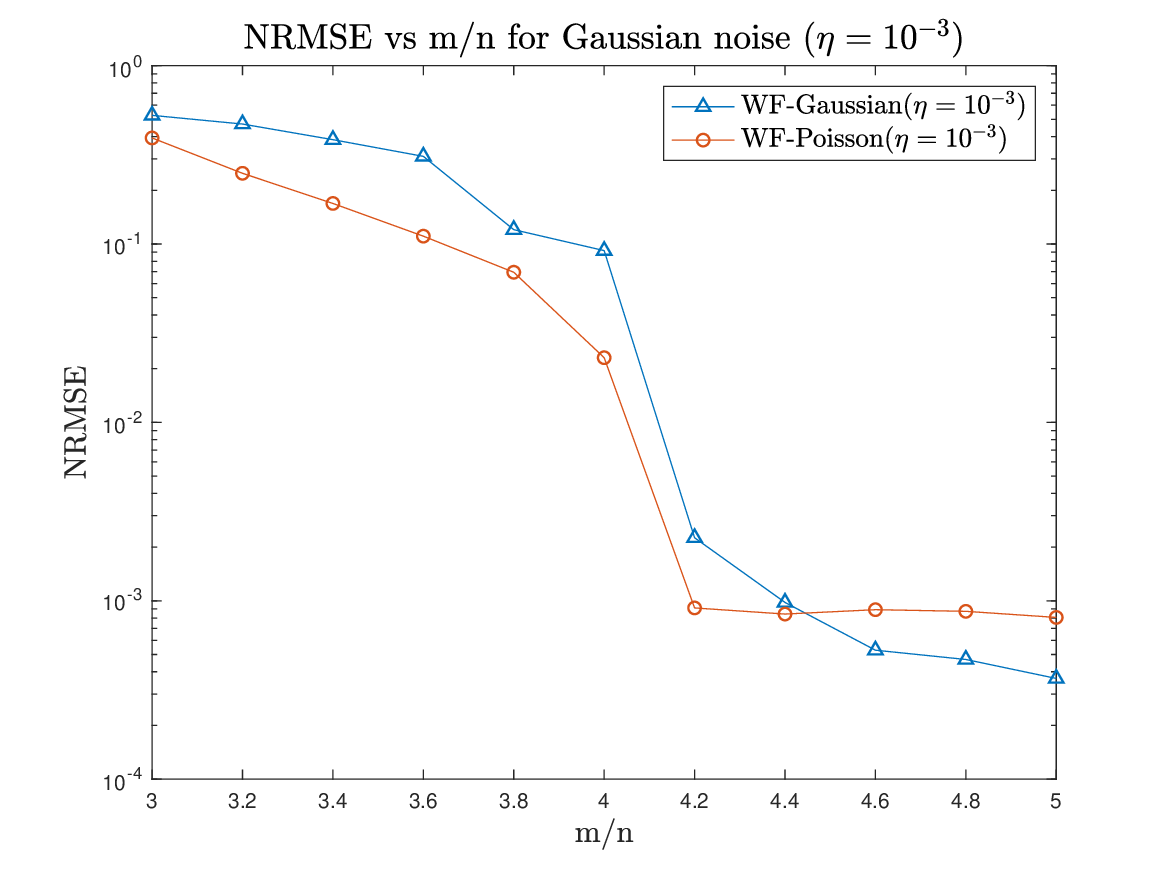}}\quad
	\subfigure[Gaussian noise ($\eta=0.1)$]{
		\includegraphics[width=0.45\textwidth]{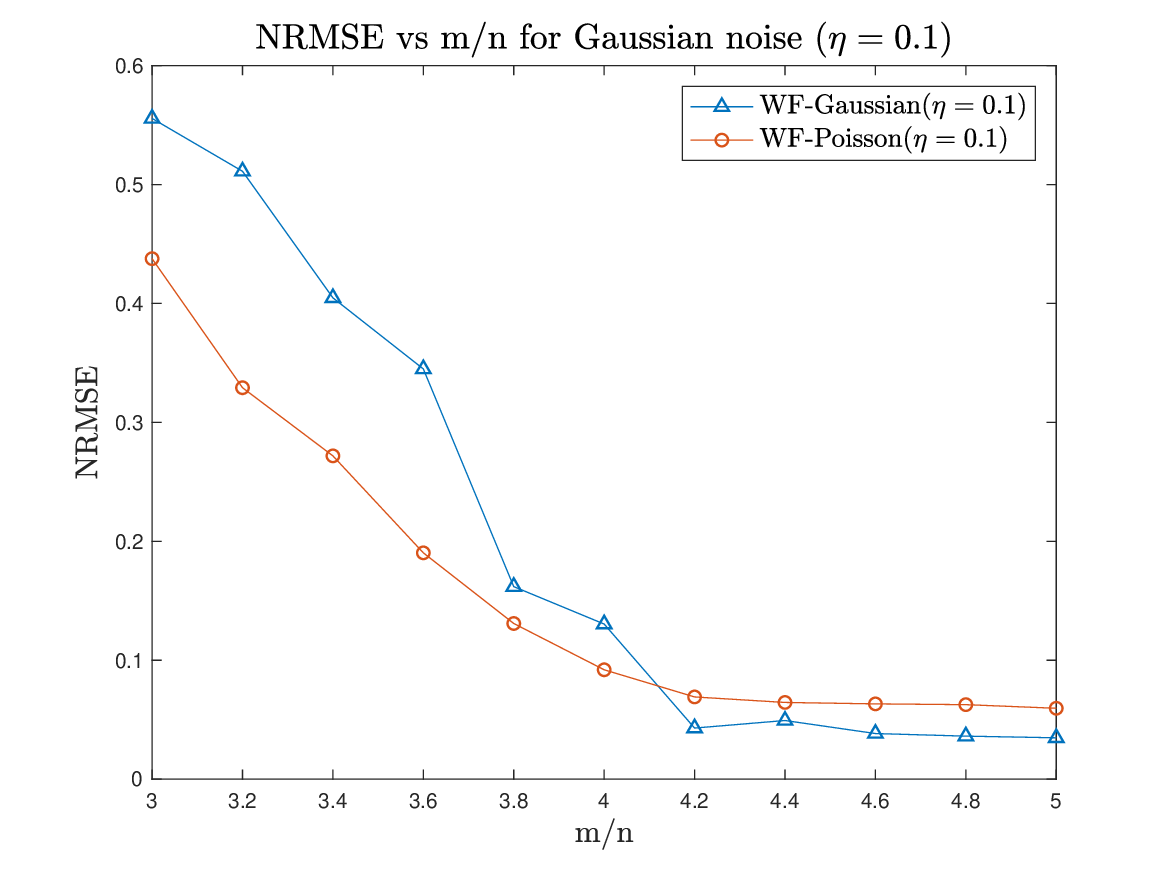}}
	\caption{Comparison of WF-Gaussian and WF-Poisson: NRMSE is recorded for different measurements under two noise types. (a) and (b) show results under Poisson noise, while (c) and (d) show results under Gaussian noise.
	}\label{compare}
\end{figure}              
\section{Conclusion}\label{conclude}

This paper provides a rigorous theoretical analysis of the Wirtinger Flow (WF) method for Poisson phase retrieval. We established that, under noiseless conditions and with an optimal number of measurements, WF achieves linear convergence to the true signal. Additionally, we demonstrated that WF remains robust and stable when dealing with bounded noise.  Building on the simplicity of WF, we proposed an incremental variant of WF method that processes one measurement at a time, significantly reducing computational cost. Our theoretical analysis showed that this incremental algorithm convergences to the true signal with high probability in the noiseless case, while maintaining performance comparable to more complex incremental methods.  

In future research, we could explore some adaptive step-size strategies and evaluate algorithm performance under more general noise models, such as Poisson-Gaussian noise.

\appendix 
\section{Local smoothness condition and local curvature condition}\label{sec_appA}
In our previous convergence analysis, we relied heavily on both the local smoothness and curvature conditions. This section focuses on presenting and proving these two key properties.
\begin{lemma}[Smoothness Condition]\label{smoothness}
	Under the same assumptions as Theorem \ref{mainresult1}, for any $ \delta>0 $, there exist constants $ C_\delta, c_\delta >0$  such that for $ m\geq C_\delta n $, we have 
	\[
	\|\nabla f(\vz)\|\leq \Big(1+\frac{1}{2\sqrt{\alpha_1}}\Big)(1+\delta)\cdot\textup{dist}(\vx,\vz)
	\]
	with probability at least $ 1-\exp(-c_\delta n) $ for any $\vz\in\C^n$. In particular, by setting $ \alpha_1= 0.8 $ and $ \delta=0.01 $, we obtain with high probability that
	\[
	\|\nabla f(\vz)\|\leq 1.58\cdot\textup{dist}(\vx,\vz).
	\]
\end{lemma}
\begin{proof}
	For any $ \vz\in\C^n $, let $ \vh = e^{-i\phi(\vz)}\vz -\vx $, so $ \|\vh\|=\dist(\vz, \vx) $. Define $ A=[\va_1,\va_2,\ldots,\va_m]^* \in\C^{m\times n}$ and $ \vv=[v_1,v_2,\ldots,v_m]\zz $, where $ v_j = \left(1-\frac{y_j}{|\va_j^*\vz|^2+b_j}\right)\va_j^*\vz $. Then $ \nabla f(\vz)=\frac{1}{m}A^*\vv $. Given that $ b_j\geq \alpha_1|\langle \va_j, \vx\rangle|^2$, we have
	\begin{align*}
		|v_j|&=\left|\frac{|\va_j^*\vz|^2-|\va_j^*\vx|^2}{|\va_j^*\vz|^2+b_j}\right||\va_j^*\vz|\\
		&\leq \frac{(|\va_j^*\vz|+|\va_j\vx|)\cdot |\va_j^*\vh|}{|\va_j^*\vz|^2+b_j}|\va_j^*\vz|\\
		&= \left(\frac{|\va_j^*\vz|^2}{|\va_j^*\vz|^2+b_j}+\frac{|\va_j^*\vz|\cdot|\va_j^*\vx|}{|\va_j^*\vz|^2+b_j}\right)|\va_j^*\vh|\\
		&\leq \Big(1+\frac{1}{2\sqrt{\alpha_1}}\Big)|\va_j^*\vh|.
	\end{align*}
	According to Lemma \ref{sub_gaussian_concentration}, for any $ \delta'>0 $ and  $ m\geq C_{\delta'}n $ with a sufficiently large constant $ C_{\delta'} $, the inequality
	\[
	\|\vv\|^2=\sum_{j=1}^{m}|v_j|^2\leq \Big(1+\frac{1}{2\sqrt{\alpha_1}}\Big)^2\sum_{j=1}^{m}|\va_j^*\vh|^2\leq\Big(1+\frac{1}{2\sqrt{\alpha_1}}\Big)^2(1+\delta')m\|\vh\|^2
	\]
	holds with probability at least $ 1-\exp(-c_{\delta'}n) $ for some $ c_{\delta'}>0 $. 
	Considering the Gaussian random matrix $ A $, for any $ \delta''>0 $ and $ m\geq C_{\delta''}n $, we have $ \|A^*\|\leq (1+\delta'')\sqrt{m} $ with probability at least $ 1-\exp(-c_{\delta''}n) $ (by Lemma \ref{sub_gaussian_concentration}). Combining these results, we get
	\begin{align*}
		\|\nabla f(\vz)\|&=\frac{1}{m}\|A^*\vv\|\leq \frac{1}{m}\|A^*\|\|\vv\|\\
		&\leq \Big(1+\frac{1}{2\sqrt{\alpha_1}}\Big)\sqrt{(1+\delta')}(1+\delta'')\|\vh\|\\
		&\leq\Big(1+\frac{1}{2\sqrt{\alpha_1}}\Big)(1+\delta)\|\vh\|
	\end{align*}
	with probability at least $ 1-\exp(-c_{\delta}n) $, provided $ m\geq C_\delta n$ for some $ C_\delta, c_\delta>0 $. Here, we choose $ 1+\delta\geq \sqrt{(1+\delta')}(1+\delta'') $ and $ C_\delta\geq \max\{C_\delta', C_\delta''\} $. 
\end{proof}
Next, we state and prove the curvature condition for the gradient.
\begin{lemma}[Curvature Condition]\label{curvature condition}
	Under the same assumptions as Theorem \ref{mainresult1}, there exist positive constants $ C$ and $c $ such that for any $ \vz\in\mathcal{S}_{\vx}(\rho) $ with $ \rho<1 $ and $ m\geq Cn $, we have
	\[
	\Re\left( \langle\nabla f(\vz), \vz-\vx e^{i\phi(\vz)}\rangle\right)\geq  l_{cur}\cdot \dist^2(\vz,\vx)
	\]
	with probability at least $ 1-\exp(-cn) $. Specially, by choosing $ \alpha_1=0.8 $, $ \alpha_2=1.2 $ and $ \rho = 1/15 $, we obtain $ l_{cur} = 0.0126 $. Detailed requirements for these parameters $ \alpha_1 $, $ \alpha_2 $ and $ \rho $ are provided in Lemma \ref{con_a_rho}.
\end{lemma}
\begin{proof}
	Without loss of generality, we assume that the target signal $\vx$ is a unite vector, i.e., $ \|\vx\|=1 $. For each $ \vz\in\C^n $, we define $ \vh = e^{-i\phi(\vz)} \vz -\vx $ and $ \tilde{\vh} := \vh/\|\vh\| $ and $ s=\|\vh\|<1 $. Given the known conditions, we have $ \Im (\vh^*\vx)=0 $ and $ \|\vh\| \leq \rho $. Using $ \vh $,  we expand:
	\begin{align*}
		\Re\big(\langle \nabla f(\vz),\, \vz-\vx e^{i\phi(\vz)}\rangle\big)
		&=\Re\big(\langle \nabla f(\vz), \, e^{i\phi(\vz)}\vh \rangle\big)	\\
		&= \frac{1}{m}\sum_{j=1}^{m}\left(1-\frac{y_j}{|\va_j^*\vz|^2 +b_j}\right)\Re\big((\va_j^*\vz)\,e^{-i\phi(\vz)}(\vh^*\va_j)\big)\\
		&=\frac{1}{m}\sum_{j=1}^{m}\frac{|\va_j^*\vz|^2 -|\va_j^*\vx|^2}{|\va_j^*\vz|^2 +b_j}\Re\big(\va_j^*(\vx+\vh)(\vh^*\va_j)\big)\\
		&=\frac{1}{m}\sum_{j=1}^{m}\frac{2\big(\Re(\vh^*\va_j\va_j^*\vx)\big)^2+ 3\Re(\vh^*\va_j\va_j^*\vx)|\va_j^*\vh|^2+|\va_j^*\vh|^4}{|\va_j^*\vx|^2 +|\va_j^*\vh|^2+2\Re(\vh^*\va_j\va_j^*\vx)+b_j}\\
		&=\frac{1}{m}\sum_{j=1}^{m}T_j,
	\end{align*}
	where $T_j=\frac{n_j}{d_j} $ denotes the $ j $-th term in the summation.
	
	To prove the conclusion, we first establish it for a fixed $ \vh $ i.e., a fixed $ \vz $ and subsequently apply a covering argument to extend the result to any $ \vz $. 
	
	\textbf{Step 1: The conclusion holds for a fixed $ \vh $.}
	
	\textbf{Case 1:} Suppose $\tilde{\vh} =c\vx$ with $\abs{c}=1$. 
	
	In this case, the condition $\Im(\tilde{\vh}^*\vx)=0$ implies that $\tilde{\vh}$ can only be $ \vx$ or $ -\vx $. Consequently, we have 
	\begin{align*}
		T_j
		&=\frac{2\big(\Re(\vh^*\va_j\va_j^*\vx)\big)^2+ 3\Re(\vh^*\va_j\va_j^*\vx)|\va_j^*\vh|^2+|\va_j^*\vh|^4}{|\va_j^*\vx|^2 +|\va_j^*\vh|^2+2\Re(\vh^*\va_j\va_j^*\vx)+b_j}\\
		&\geq\frac{\big(2\pm 3\|\vh\|+\|\vh\|^2\big)|\va_j^*\vx|^4}{(1+\alpha_2+\|\vh\|^2+2\|\vh\|)|\va_j^*\vx|^2}\|\vh\|^2\\
		&\geq \frac{\big(2- 3s+s^2\big)|\va_j^*\vx|^2}{\big((1+s)^2+\alpha_2\big)}\|\vh\|^2
	\end{align*}
	Then, by Lemma \ref{sub_gaussian_concentration}, for any $ \delta>0 $, there exist positive constants $ C_\delta$ and $c_\delta  $ such that when $ m\geq C_\delta n $, 
	\begin{equation}
		\begin{aligned}\label{specialcase}
			\Re\big(\langle \nabla f_{\vepsilon}(\vz),\, \vz-\vx e^{i\phi(\vz)}\rangle\big)
			& =\frac{1}{m}\sum_{j=1}^{m} T_j \\
			&\geq \frac{1}{m}\sum_{j=1}^{m} \frac{\big(2- 3s+s^2\big)}{\big((1+s)^2+\alpha_2\big)}|\va_j^*\vx|^2\|\vh\|^2\\
			&\geq \frac{\big(2- 3s+s^2\big)}{\big((1+s)^2+\alpha_2\big)}(1-\delta)\|\vh\|^2.
		\end{aligned}
	\end{equation}
	This inequality holds with probability greater than $ 1-\exp(-c_\delta m) $. 	
	\textbf{Case 2:}  Consider that $ \tilde{\vh}\neq \pm \vx $.
	
	Since the Gaussian random measurement $ \va $ is rotationally invariant, we have:
	\begin{equation}  \label{pro3}
		\PP(|\va^*\vx|> |\va_j^*\tilde{\vh}|)=\PP(|\va^*\vx|\leq |\va_j^*\tilde{\vh}|)=1/2.
	\end{equation}
	For each index set $ I\subseteq \{1,2,\ldots,m\} $, we define the corresponding event:
	$$
	\E_I:=\bigl\{ |\va_j^*\vx|>|\va_j^*\tilde{\vh}|, \,\,\forall j\in I ;\,\, |\va_k^*\vx|\leq|\va_k^*\tilde{\vh}|, \,\,\forall k\in I^c\bigr\} .
	$$
	According to (\ref{pro3}), the event $ \E_I  $ occurs with probability $ 1/2^m $.  We assume that  $ I_0 $ is an index set satisfying $ \frac{m}{4}\leq|I_0|\leq \frac{3m}{4} $. Under event $ \E_{I_0} $, $\Re\big(\langle \nabla f_{\vepsilon}(\vz),\, \vz-\vx e^{i\phi(\vz)}\rangle\big) $
	can be divided into two  groups:
	\begin{align*}
		m\,\Re\big(\langle \nabla f(\vz), \, \vz-e^{i\phi(\vz)}\vx \rangle\big)
		=\sum_{j\in I_0}T_j+\sum_{k\in I_0^c}T_k.
	\end{align*}
	We now proceed to analyze the lower bounds for each group. First, we establish bounds for the denominators $ d_j $, $ j=1,\ldots,m$.  For $ j\in I_0 = \big\{j\,:\, |\va_j^*\vx|>|\va_j^*\tilde{\vh}|\big\} $,  we have
	\begin{equation}\label{1upperbound}
		\begin{aligned}
			d_j & = |\va_j^*\vx|^2 +|\va_j^*\vh|^2+2\Re(\vh^*\va_j\va_j^*\vx)+b_j\\
			&\leq(1+\|\vh\|^2+2\|\vh\|+\alpha_2)|\va_j^*\vx|^2= U_1|\va_j^*\vx|^2.
		\end{aligned}
	\end{equation}
	where $ U_1:= (1+s)^2+\alpha_2 $. Here we use the fact that $ \|\vh\|\leq \rho $.
	On the other hand, we have the following lower bound:
	\begin{equation}\label{1lowerbound}
		\begin{aligned}
			d_j & = |\va_j^*\vx|^2 +|\va_j^*\vh|^2+2\Re(\vh^*\va_j\va_j^*\vx)+b_j\\
			&\geq|\va_j^*\vx|((1+\alpha_1)|\va_j^*\vx|-2|\va_j^*\vh|)+|\va_j^*\vh|^2\\
			&\geq\big((1-\|\vh\|)^2+\alpha_1\big)|\va_j^*\tilde{\vh}|^2= L_1|\va_j^*\tilde{\vh}|^2
		\end{aligned}
	\end{equation}
	where $ L_1:= (1-s)^2+\alpha_1.$	
	Similarly, for $ k\in  I_0^c = \big\{k\,:|\va_k^*\vx|\leq|\va_k^*\tilde{\vh}|\big\} $, we have
	\begin{equation}\label{2upperbound}
		\begin{aligned}
			d_k &=|\va_k^*\vx|^2 +|\va_k^*\vh|^2+2\Re(\vh^*\va_k\va_k^*\vx)+b_j\\
			&\leq \big(1+\alpha_2+2\|\vh\|+\|\vh\|^2\big)|\va_k^*\tilde{\vh}|^2= U_2|\va_k^*\tilde{\vh}|^2
		\end{aligned}
	\end{equation}
	where $ U_2:=(s+1)^2+\alpha_2=U_1 $ and
	\begin{equation}\label{2lowerbound}
		\begin{aligned}
			d_k &=|\va_k^*\vx|^2 +|\va_k^*\vh|^2+2\Re(\vh^*\va_k\va_k^*\vx)+b_j\\
			&\geq\alpha_1|\va_k^*\vx|^2+(|\va_k^*\vx|-|\va_k^*\vh|)^2=L_2|\va_k^*\vx|^2
		\end{aligned}
	\end{equation}
	where $ L_2 := \alpha_1$.
	
	Based on the bounds established in (\ref{1upperbound}) and  (\ref{1lowerbound}), and utilizing Lemma \ref{concentration}, we analyze the lower bound of $  \sum_{j\in I_0}T_j $. For a sufficiently small $ \delta>0 $, when $ |I_0| \geq C_1(\delta)n $ with probability at least $ 1-\exp\big(-c_1(\delta)\cdot|I_0|\big) $, we have
	\begin{equation}\label{i}
		\begin{aligned}
			\sum_{j\in I_0}T_j
			&= \sum_{j\in I_0}\Bigg(\frac{\big(\sqrt{2}\Re(\vh^*\va_j\va_j^*\vx)+\frac{3}{2\sqrt{2}}|\va_j^*\vh|^2\big)^2}{d_j} - \frac{|\va_j^*\vh|^4}{8d_j}\Bigg)\\
			& \geq \sum_{j\in I_0} \Bigg(\frac{2\big(\Re(\vh^*\va_j\va_j^*\vx)\big)^2- 3|\vh^*\va_j\va_j^*\vx||\va_j^*\vh|^2}{U_1|\va_j^*\vx|^2} - \frac{|\va_j^*\vh|^4}{8L_1|\va_j^*\tilde{\vh}|^2}\Bigg)\\
			&\geq\sum_{j\in I_0}\Bigg(\frac{2\|\vh\|^2}{U_1}\frac{\big(\Re(\tilde{\vh}^*\va_j\va_j^*\vx)\big)^2}{|\va_j^*\vx|^2} - \frac{3\|\vh\|^3}{U_1}|\va_j^*\tilde{\vh}|^2-\frac{|\va_j^*\vh|^2}{8L_1}\|\vh\|^2\Bigg)\\
			&\geq|I_0|\cdot\|\vh\|^2\Bigg(\frac{2}{U_1}\Big(\frac{1}{8}+\frac{7}{32}\Re^2(\tilde{\vh}^*\vx)\Big)-\frac{3}{2U_1}\|\vh\|-\frac{\|\vh\|^2}{16L_1}-\frac{\delta}{4}\Bigg)\\
			&\geq|I_0|\cdot\|\vh\|^2\Big(\phi_1+ \frac{7}{16U_1}\Re^2(\tilde{\vh}^*\vx)\Big),
		\end{aligned}
	\end{equation}	
	where $ \phi_1 := \frac{1-6s}{4U_1}-\frac{s^2}{16L_1}-\frac{\delta}{4}$. Here, the third inequality is derived from Lemma \ref{concentration}. 
	
	Similarly, according to (\ref{2upperbound}), (\ref{2lowerbound}) and Lemma \ref{concentration}, when $ |I_0^c|\geq C_2(\delta)n $, with probability at least $ 1-\exp\big(-c_2(\delta)\cdot|I_0^c|\big) $, we have
	\begin{equation}\label{ic}
		\begin{aligned}
			\quad	&\sum_{k\in I_0^c} T_k= \sum_{k\in I_0^c} \Bigg( \frac{\big(\frac{3}{2}\Re(\vh^*\va_k\va_k^*\vx)+|\va_k^*\vh|^2\big)^2}{d_k}-\frac{\big(\Re(\vh^*\va_k\va_k^*\vx)\big)^2}{4d_k} \Bigg)\\
			&\geq \sum_{k\in I_0^c}\Bigg(\frac{\frac{9}{4}\big(\Re(\vh^*\va_k\va_k^*\vx)\big)^2+ 3\Re(\vh^*\va_k\va_k^*\vx)|\va_k^*\vh|^2+|\va_k^*\vh|^4}{U_2|\va_k^*\tilde{\vh}|^2}-\frac{\big(\Re(\vh^*\va_k\va_k^*\vx)\big)^2}{4L_2|\va_k^*\vx|^2}\Bigg )\\
			&\geq|I_0^c|\cdot\|\vh\|^2 \Bigg(\frac{9}{4U_2}\Big(\frac{1}{8}+\frac{7}{32}\Re^2(\tilde{\vh}^*\vx)\Big)+\frac{3\|\vh\|}{2U_2}\Re(\tilde{\vh}^*\vx)+\frac{\|\vh\|^2}{2U_2}-\frac{1}{4L_2}\Big(\frac{3}{8}+\frac{9}{32}\Re^2(\tilde{\vh}^*\vx)\Big)-\frac{\delta}{4}\Bigg)\\
			&\geq|I_0^c|\cdot \|\vh\|^2\Bigg(\phi_2+\psi\cdot\Re^2(\tilde{\vh}^*\vx)+\frac{3\|\vh\|}{2U_2}\Re(\tilde{\vh}^*\vx)\Bigg),
		\end{aligned}
	\end{equation}
	where $ \phi_2= \frac{9+16s^2}{32U_2}-\frac{3}{32L_2}-\frac{\delta}{4}$ and $ \psi=\frac{63}{128U_2}-\frac{9}{128L_2} $.
	Here the second inequality follows from Lemma \ref{concentration}.
	
	Choose a sufficiently small constant $ \delta$. Then, for a sufficiently large constant $ C\geq 4\max\{C_1(\delta), C_2(\delta)\} $, as long as $ m\geq Cn $, we have that the index set $I_0$ satisfies $ 3m/4\geq|I_0|\geq m/4 \geq C_1(\delta)n $ and $ 3m/4\geq|I_0^c|\geq m/4 \geq C_2(\delta)n $. Combining inequalities (\ref{i}) and (\ref{ic}), with probability at least $ (1-\exp(-c_3m))/2^m $,  we obtain
	\begin{equation}\label{generalcase}
		\begin{aligned}
			\quad&\Re\big(\langle \nabla f(\vz),\, \vz-\vx e^{i\phi(\vz)}\rangle\big)
			=\frac{1}{m}\Big(\sum_{j\in I_0}T_j+\sum_{k\in I_0^c}T_k\Big)\\
			&\geq\frac{1}{m}\|\vh\|^2\bigg(|I_0|\cdot \Big(\phi_1+ \frac{7}{16U_1}\Re^2(\tilde{\vh}^*\vx)\Big)+|I_0^c|\cdot\Big(\phi_2+\psi\cdot\Re^2(\tilde{\vh}^*\vx)+\frac{3\|\vh\|}{2U_2}\Re(\tilde{\vh}^*\vx)\Big) \bigg)\\
			&\geq\frac{1}{m}\|\vh\|^2\left(\frac{m}{4}\cdot (\phi_1+\phi_2)+\frac{m}{4}\cdot \left(\frac{7}{16U_1}+\psi\right)\Re^2(\tilde{\vh}^*\vx)-\frac{3m}{4}\cdot\frac{3\|\vh\|}{2U_2}|\Re(\tilde{\vh}^*\vx)|\right)\\
			&\geq\frac{\|\vh\|^2}{4}(\phi_1+\phi_2-\varphi),
		\end{aligned}
	\end{equation}
	where $ \varphi=\frac{(9s/(4U_2))^2}{7/(16U_1)+\psi}=\frac{81U_1s^2}{U_2^2(7+16U_1\psi)} $. 
	One sufficient condition for the second inequality  to hold is:
	\[
	\phi_1>0,\,\, \phi_2>0,\,\, \psi>0.
	\]
	We assert that these conditions indeed hold, with detailed analysis provided in Lemma  \ref{con_a_rho}.
	
	The number of index sets $I$  satisfying $ \frac{m}{4}\leq |I|\leq \frac{3m}{4}$ is $ \sum_{k=m/4}^{3m/4} {m\choose k}$.  Therefore, for a fixed $ \tilde{\vh} $ where $ \tilde{\vh}\neq\pm \vx $, the inequality (\ref{generalcase})
	holds with probability greater than $ \sum_{k=m/4}^{3m/4} {m\choose k}(1-\exp(-c_3m))/2^m \geq 1-\exp(-c_4m)$. 
	
	Combining (\ref{specialcase}) and (\ref{generalcase}) and defining 
	\[
	\hat{l}_{cur}:=\min\Big\{\frac{(\phi_1+\phi_2-\varphi)}{4},\frac{\big(2- 3s+s^2\big)}{\big((1+s)^2+\alpha_2\big)}(1-\delta)\Big\}>0,
	\] 
	we conclude that for a fixed vector $ \vz $,
	\begin{equation}\label{aimfunc}
		\Re\big(\langle \nabla f(\vz),\, \vz-\vx e^{i\phi(\vz)}\rangle\big)
		\geq \hat{l}_{cur}\|\vh\|^2
	\end{equation}
	holds with probability at least $ 1-\exp(-c_5m)$ provided enough measurements. This completes the proof that (\ref{aimfunc}) holds for a fixed $ \vz $, i.e., a fixed $ \tilde{\vh} $ and a fixed value $ \|\vh\|=s\leq \rho $.
	
	\textbf{Step 2: Extension to all vectors.}
	
	Observe that
	\begin{align*}
		\Re\big(\langle \nabla f(\vz),\, \vz-\vx e^{i\phi(\vz)}\rangle\big)
		=\Re\big(\langle \nabla f (\vx+\vh),\, \vh \rangle\big)=\Re\big(\langle \nabla f (\vx+\|\vh\|\tilde{\vh}),\, \|\vh\|\tilde{\vh} \rangle\big)=\Re\big(\langle \nabla f (\vx+s\tilde{\vh}),\, s\tilde{\vh} \rangle\big).
	\end{align*}
	Thus, for any unit vectors $ \tilde{\vh}_1,\, \tilde{\vh}_2\in \C^n $, we have
	\begin{align*}
		&\big|\Re\big(\langle \nabla f(\vx+s{\vh}_1), s\tilde{\vh}_1 \rangle\big)-\Re\big(\langle \nabla f (\vx+s\tilde{\vh}_2), s\tilde{\vh}_2 \rangle\big)\big| \\
		& \leq\big|\Re\big(\langle \nabla f (\vx+s\tilde{\vh}_1), s(\tilde{\vh}_1-\tilde{\vh}_2)\rangle\big)\big|
		+\big|\Re\big(\langle \nabla f (\vx+s\tilde{\vh}_1) - \nabla f (\vx+s\tilde{\vh}_2), s\tilde{\vh}_2 \rangle\big)\big|\\
		&\leq \big(s\|\nabla f(\vx+s\tilde{\vh}_1)\| + c_1s^2\big)\cdot\|\tilde{\vh}_1-\tilde{\vh}_2\| \\
		&\leq (2s^2\|\tilde{\vh}_1\| + c_1s^2\|\tilde{\vh}_2\|)\cdot \|\tilde{\vh}_1-\tilde{\vh}_2\|\\
		&<(2+c_1) s^2 \cdot\|\tilde{\vh}_1-\tilde{\vh}_2\|.
	\end{align*}
	Here, $ \xi\in \C^n $ and the third inequality follows from Lemma \ref{smoothness} and Lemma \ref{hessian}.
	
	Thus, for any $ \tilde{\vh}_1,\, \tilde{\vh}_2\in \C^n $ with $ \|\tilde{\vh}_1\|=\|\tilde{\vh}_2\|=1$ and $ \|\tilde{\vh}_1 - \tilde{\vh}_2\|\leq \eta:=\frac{\delta}{2(2+c_1)} $ with $ \delta $ sufficiently small, we have
	\begin{equation}\label{covering1}
		\Re\big(\langle \nabla f(\vx+s\tilde{\vh}_1), s\tilde{\vh}_1 \rangle\big)-\Re\big(\langle \nabla f (\vx+s\tilde{\vh}_2), s\tilde{\vh}_2 \rangle\big)\geq -\s^2\delta.
	\end{equation}
	Let $ \NN_\eta $ be an $ \eta $-net for the unit sphere of $ \C^n $ with cardinality $ |\NN_\eta|\leq (1+2/\eta)^{2n}$. Then for all $\tilde{\vh }\in \NN_\eta$ and fixed $ s\leq\rho $, when $m\geq (C_2\cdot \eta^{-2}\log \eta^{-1})n$, with probability at least $ 1-|\NN_\eta|\exp(-c_4 n) $ we have
	\begin{align}\label{covering2}
		\Re\big(\langle \nabla f(\vz),\, \vz-\vx e^{i\phi(\vz)}\rangle\big)
		&= \Re\big(\langle \nabla f (\vx+s\tilde{\vh}),\,s\tilde{\vh}\rangle\big)\geq \hat{l}_{cur} s^2.
	\end{align}
	For any $ \tilde{\vh} $ with $ \|\vh\|=1 $, there exists $ \tilde{\vh}_1\in\NN_\eta $ such that $ \| \tilde{\vh}- \tilde{\vh}_1\|\leq \eta $. Combining (\ref{covering1}) and (\ref{covering2}), we conclude that
	\begin{align*}
		\Re\big(\langle \nabla f(\vx+s\tilde{\vh}),\, s\tilde{\vh}\rangle\big)
		\geq (\hat{l}_{cur}-\delta/2)s^2.
	\end{align*}
	Applying a similar covering number argument over $ s\leq  \rho $, we
	can further conclude that for all $ \tilde{\vh} $ and $ s $,
	\begin{align*}
		\Re\big(\langle \nabla f(\vz),\, \vz-\vx e^{i\phi(\vz)}\rangle\big)=\Re\big(\langle \nabla f(\vx+s\tilde{\vh}),\, s\tilde{\vh}\rangle\big)
		\geq (\hat{l}_{cur}-\delta)\|\vh\|^2.
	\end{align*}
	holds with probability at least $ 1-\exp(-cn) $, provided $ m\geq Cn $ with a sufficiently large constant $ C $. Then the theorem is proved by setting $l_{cur}:=\hat{l}_{cur}-\delta  $.		
\end{proof}
\section{Useful Lemmas}\label{sec_appB}
In this section, we provide some useful lemmas that were applied in proving Lemmas \ref{smoothness} and \ref{curvature condition}.

\begin{lemma}[\cite{candes2013phaselift} Lemma 3.1 ]\label{sub_gaussian_concentration}
	Let $ \va_1,\va_2,\ldots,\va_m \in\C^n$ be i.i.d. Gaussian random measurements. Fix any $ \delta $ in $ (0,1/2) $ and assume $ m\geq 20\delta^{-2}n $. Then for all unit vectors $ \vu\in\C^n $,
	\[
	1-\delta\leq\frac{1}{m}\sum_{j=1}^{m}|\va_j^*\vu|^2\leq 1+\delta
	\]
	holds with probability at least $ 1-\exp(-m t^2/2) $, where $ \delta/4 = t^2+t $.
\end{lemma}

\begin{lemma}[\cite{gao2020perturbed} Lemma A.3]\label{concentration}
	Let $ \va_1, \va_2,\ldots,\va_m \in \C^n$  be i.i.d. Gaussian random measurements. Let $ \vx\in\C^n $ and $ \tilde{\vh}\in\C^n $ be two fixed vectors with $ \|\vx\|=\|\tilde{\vh}\|=1 $, $ \Im(\tilde{\vh}^*\vx)=0 $ and $ \tilde{\vh}\neq\pm \vx $. For any $ \delta>0 $, there exist positive constants $C_\delta, c_\delta>0$ such that for any $m \geq C_\delta n$ the inequalities
	\begin{align}\label{con1}
		\bigg|\frac{1}{m}\sum_{j=1}^{m}	 \Re(\tilde{\vh}^*\va_j\va_j^*\vx)\cdot I_{\{|\va_j^*\vx|>|\va_j^*\tilde{\vh}|\}}-\frac{1}{2}\Re(\tilde{\vh}^*\vx)\bigg|	 \leq \delta,
	\end{align}
	\begin{align}\label{con2}
		\frac{1}{2}-\delta\leq \frac{1}{m}\sum_{j=1}^{m}|\va_j^*\vx|^2\cdot I_{\{|\va_j^*\vx|>|\va_j^*\tilde{\vh}|\}}\leq \frac{3}{4}+\delta,
	\end{align}
	\begin{align}\label{con2'}
		\frac{1}{4}-\delta\leq \frac{1}{m}\sum_{j=1}^{m}|\va_j^*\vx|^2\cdot I_{\{|\va_j^*\vx|\leq|\va_j^*\tilde{\vh}|\}}\leq \frac{1}{2}+\delta,
	\end{align}
	\begin{equation}
		\begin{aligned}\label{con3}
			\frac{1}{8}+\frac{7}{32}\Re^2(\tilde{\vh}^*\vx)-\delta\leq\frac{1}{m}\sum_{j=1}^{m}	 \frac{\big(\Re(\tilde{\vh}^*\va_j\va_j^*\vx)\big)^2}{|\va_j^*\vx|^2}\cdot I_{\{|\va_j^*\vx|>|\va_j^*\tilde{\vh}|\}}\leq\frac{1}{4}+\frac{1}{4}\Re^2(\tilde{\vh}^*\vx)+\delta
		\end{aligned}
	\end{equation}
	and
	\begin{equation}
		\begin{aligned}\label{con4}
			\frac{1}{4}+\frac{1}{4}\Re^2(\tilde{\vh}^*\vx)-\delta\leq\frac{1}{m}\sum_{j=1}^{m}	 \frac{\big(\Re(\tilde{\vh}^*\va_j\va_j^*\vx)\big)^2}{|\va_j^*\vx|^2})\cdot I_{\{|\va_j^*\vx|\leq|\va_j^*\tilde{\vh}|\}}\leq\frac{3}{8}+\frac{9}{32}\Re^2(\tilde{\vh}^*\vx)+\delta
		\end{aligned}
	\end{equation}
	hold with probability at least $ 1-\exp(-c_\delta m) $.
\end{lemma}
The following lemma provides an upper bound for the operator norm of $  \nabla^2 f(\vz)$.

\begin{lemma}\label{hessian}
	Suppose $ \vh_1 $ and $ \vh_2 $ are two vectors with $ \|\vh_1\|=\|\vh_2\|=s$.
	There exist constants $C',c',c_1>0$ such that when $ m\geq C'n $,
	$|\Re\big(\langle\nabla f(\vx+\vh_1) - \nabla f(\vx+\vh_2),\vh_2\rangle\big)|\leq c_1s\|\vh_1-\vh_2\|  $ holds with probability at least $ 1-\exp(-c'm) $.
\end{lemma}
\begin{proof}
	Recall that
	\begin{align*}
		\nabla f(\vz) &:= \left(\frac{\partial f(\vz,\overline{\vz})}{\partial \vz}\Big|_{\overline{\vz} = \text{constant}}\right)^* = \frac{1}{m}\sum_{j=1}^{m}\left(1-\frac{y_j}{|\va_j^*\vz|^2 +b_j}\right)\va_j\va_j^*\vz.
	\end{align*}
	Then consider 
	\begin{align*}
		&|\Re\big(\langle\nabla f(\vx+\vh_1) - \nabla f(\vx+\vh_2),\vh_2\rangle\big)|\\
		&=\left|\frac{1}{m}\sum_{j=1}^{m}\left(\Big(1-\frac{y_j}{|\va_j^*(\vx+\vh_1)|^2 +b_j}\Big)\Re\big(\vh_2^*\va_j\va_j^*(\vx+\vh_1)\big)-\Big(1-\frac{y_j}{|\va_j^*(\vx+\vh_2)|^2 +b_j}\Big)\Re\big(\vh_2^*\va_j\va_j^*(\vx+\vh_2)\big)\right)\right|
	\end{align*}
	Here we define a function $ G(t) $ as
	\[
	G(t)=\frac{1}{m}\sum_{j=1}^{m}\Big(1-\frac{y_j}{|\va_j^*(\vx+\vh_2+t(\vh_1-\vh_2))|^2 +b_j}\Big)\Re\big(\vh_2^*\va_j\va_j^*(\vx+\vh_2+t(\vh_1-\vh_2))\big)
	\]
	Then the problem transformed to estimate $  |G(1)- G(0)| $. By setting $ \vz_t=\vx+\vh_2+t(\vh_1-\vh_2) $ and simple calculations, we have 
	\begin{align*}
		G'(t)&=\frac{1}{m}\sum_{j=1}^{m}\left(\Big(1-\frac{y_j}{|\va_j^*\vz_t|^2 +b_j}\Big)\Re\big(\vh_2^*\va_j\va_j^*(\vh_1-\vh_2)\big)+\frac{2y_j\Re(\vh_2^*\va_j\va_j^*\vz_t)\cdot\Re\big((\vh_1-\vh_2)^*\va_j\va_j^*\vz_t\big)}{(|\va_j^*\vz_t|^2 +b_j)^2}\right)\\
		&=\frac{1}{m}\sum_{j=1}^{m}\Bigg(\Big(1-\frac{y_j}{|\va_j^*\vz_t|^2 +b_j}\Big)\Re\big(\vh_2^*\va_j\va_j^*(\vh_1-\vh_2)\big)+2y_j\frac{|\va_j^*\vz_t|^2-\Re(\vx^*\va_j\va_j^*\vz_t)}{(|\va_j^*\vz_t|^2 +b_j)^2}\Re\big((\vh_1-\vh_2)^*\va_j\va_j^*\vz_t\big)\\
		&\quad\quad -\frac{2y_j}{(|\va_j^*\vz_t|^2 +b_j)^2}\Re^2\big((\vh_1-\vh_2)^*\va_j\va_j^*\vz_t\big)\Bigg)\\
		&=G_1+G_2+G_3
	\end{align*}
	with $ G_i,\,i=1,2,3 $ defined to simplify the expression. Note that 
	$ \|\vz_t\| \leq (1+s)$.
	Then according to Lemma \ref{sub_gaussian_concentration}, when $ m\geq C'n $, with probability at least $ 1-\exp(-c'n) $, we have
	\begin{align*}
		|G_1|&\leq\max_j\Big|1-\frac{y_j}{|\va_j^*\vz_t|^2 +b_j}\Big|\cdot\Big|\frac{1}{m}\sum_{j=1}^{m}\Re\big(\vh_2^*\va_j\va_j^*(\vh_1-\vh_2)\big)\Big|\\
		&\leq 2\frac{1+\alpha_1+\alpha_2}{\alpha_1}s\|\vh_1-\vh_2\|,
	\end{align*}
	\begin{align*}
		|G_2|&\leq\max_j2y_j\frac{|\va_j^*\vz_t|^2+|\vx^*\va_j\va_j^*\vz_t|}{(|\va_j^*\vz_t|^2 +b_j)^2}\cdot\Big|\frac{1}{m}\sum_{j=1}^{m}\Re\big((\vh_1-\vh_2)^*\va_j\va_j^*\vz_t\big)\Big|\\
		&\leq 4\frac{1+\alpha_2}{\alpha_1}(1+\frac{1}{2\sqrt{\alpha_1}})(1+s)\|\vh_1-\vh_2\|
	\end{align*}
	and 
	\begin{align*}
		|G_3|&\leq\frac{1}{m}\sum_{j=1}^{m}\frac{2y_j|\va_j^*\vz_t|^2}{(|\va_j^*\vz_t|^2 +b_j)^2}|\va_j^*(\vh_1-\vh_2)|^2\\
		&\leq 8\frac{1+\alpha_2}{\alpha_1}s\|\vh_1-\vh_2\|
	\end{align*}
	Then straightforwardly we have the following estimate
	\begin{align*}
		|G'(t)|\leq c_{1}s\|\vh_1-\vh_2\|
	\end{align*}
	where $ c_1 $ is a constant that depends on $ \alpha_1,\alpha_2 $.
\end{proof}
To derive the local curvature condition established in Lemma \ref{curvature condition}, we must ensure certain relationships hold between parameters $\alpha_1,\alpha_2$ and $\rho$. The following lemma provides an analysis of the constraints on these parameters.
\begin{lemma}\label{con_a_rho}
	
	Let $\alpha_1>0$, $\alpha_2>0$ and $\rho>0$ be constants associated with the local curvature conditions in Lemma \ref{curvature condition}. To satisfy the curvature condition, the parameters must fulfill the following relationships:
	$$
	\begin{cases}
		\alpha_1<\alpha_2<3\alpha_1-(\rho+1)^2, & \rho\in(0,\rho_1]\\
		\alpha_1<\alpha_2<(3-\frac{(3-t_2)(\rho-\rho_1)}{1/6-\rho_1})\alpha_1-(\rho+1)^2, & \rho\in(\rho_1,\rho_2],
	\end{cases} 	
	$$
	where $ \rho_1=\frac{2\sqrt{76015}-276}{2477}\approx 0.11119,\,
	t_2=\frac{8075-\sqrt{45678865}}{990}\approx 1.32968,\,
	\rho_2=\frac{2\sqrt{39}-12}{3}\approx 0.16333. $
	Under these conditions, the following quantities remain positive:
	\[
	\Phi_1>0,\,\Phi_2>0,\,\psi>0,\,\Phi_1+\Phi_2-\varphi>0,
	\]
	where the terms are defined as
	\begin{align*}
		\Phi_1&=\frac{1-6\rho}{4U}-\frac{\rho^2}{16L_1},\quad
		\Phi_2=\frac{9+16\rho^2}{32U}-\frac{3}{32L_2},\\
		\psi&=\frac{63}{128U}-\frac{9}{128L_2},\quad
		\varphi=\frac{81U\rho^2}{U^2(7+16U\psi)},
	\end{align*}
	with $ 	U=(1+\rho)^2+\alpha_2,\,L_1=(1-\rho)^2+\alpha_1,\,L_2=\alpha_1$.
\end{lemma}

\begin{proof}
	In this proof, we validate the conditions on $\alpha_1$, $\alpha_2$ and $\rho$ to ensure the positivity of critical terms, which support the local curvature condition required in Lemma \ref{curvature condition}. Each inequality corresponds to verifying that specific terms remain positive over a certain range of $\rho$.		
	\begin{itemize}
		\item Positivity of $\psi$.
		
		Since $\psi>0$ is equivalent to $\alpha_2<7\alpha_1-(1+\rho)^2$, and because we know $\alpha_2<3\alpha_1-(1+\rho)^2<7\alpha_1-(1+\rho)^2$ for all $\rho\in (0,\rho_2]$, this condition is satisfied, hence $\psi>0$ holds.\\
		
		\item Positivity of $\Phi_2$.
		
		We see that $\Phi_2>0$ is equivalent to $\alpha_2<(3+\frac{16}{3}\rho^2)\alpha_1-(1+\rho)^2$. Given that 
		$\alpha_2<3\alpha_1-(1+\rho)^2<(3+\frac{16}{3}\rho^2)\alpha_1-(1+\rho)^2$ for all $\rho\in (0,\rho_2]$,  $\Phi_2>0$ holds as well.\\
		
		\item Positivity of $\Phi_1$.
		
		For $\Phi_1>0$, it is required that $U<\frac{4(1-6\rho)L_1}{\rho^2}$. Given $\alpha_2<3\alpha_1-(1+\rho)^2$, it suffices to show that $3\alpha_1<\frac{4(1-6\rho)L_1}{\rho^2}$, which can be rewritten as $(3\rho^2+24\rho-4)\alpha_1<4(1-6\rho)(1-\rho)^2$. This inequality holds for
		$\rho\in(0,\rho_2]$ as  $3\rho^2+24\rho-4\leq 0$ in this interval. Thus 
		$\Phi_1>0$ is satisfied.\\
		
		\item Positivity of $\Phi_1+\Phi_2-\varphi$.
		
		We reformulate the condition $\Phi_1+\Phi_2-\varphi>0$ to 
		\begin{equation}\label{eq:r1}
			9(2\rho^2L_2+3L_1)\big(\frac{U}{L_2}\big)^2-[(510-432\rho+144\rho^2)L_1+238\rho^2L_2]\frac{U}{L_2}+(2023-5712\rho-18832\rho^2)L_1>0.
		\end{equation}	
		Since $\alpha_2<3\alpha_1-(\rho+1)^2$ implies $\frac{U}{L_2}<3$, we consider the minimum value of this quadratic equation in $\frac{U}{L_2}$, which occurs at 
		$$\frac{(510-432\rho+144\rho^2)L_1+238\rho^2L_2}{18(2\rho^2L_2+3L_1)}.$$
		Since $$\frac{(510-432\rho+144\rho^2)L_1+238\rho^2L_2}{18(2\rho^2L_2+3L_1)}-3=\frac{(348-432\rho+144\rho^2)L_1+130\rho^2L_2}{18(2\rho^2L_2+3L_1)}>0,$$
		We have that the minimal value of \eqref{eq:r1} is larger than the value at $\frac{U}{L_2}=3$ when $\rho\in(0,\rho_1]$ and the value at $\frac{U}{L_2}=3-\frac{(3-t_2)(\rho-\rho_1)}{1/6-\rho_1}$ when $\rho\in(\rho_1,\rho_2]$.
		
		For $\rho\in(0,\rho_1]$, substituting $\frac{U}{L_2}=3$ simplifies the expression to 
		$$(736-4416\rho-19264\rho^2)(1-\rho)^2+(736-4416\rho-19816\rho^2)\alpha_1.$$
		Since $736-4416\rho-19816\rho^2\geq 0$ for all $\rho\in(0,\rho_1]$, the expression is positive.
		
		For $\rho\in(\rho_1,\rho_2]$, substituting $\frac{U}{L_2}=3-\frac{(3-t_2)(\rho-\rho_1)}{1/6-\rho_1}$, we verify that the resulting fourth-order polynomial in $\rho$ has roots $\rho_1$, $1/6$, approximately $-0.9853$ and $0.42427$.
		$\rho_1$ and $1/6$ are the roots in the middle, thus \eqref{eq:r1} holds for all $\rho\in(\rho_1,1/6)$. Therefore, we have $\Phi_1+\Phi_2-\varphi>0$ for all $\rho\in(0,\rho_2]$.
	\end{itemize}				
\end{proof}
{
	\bibliographystyle{plain}
	\bibliography{poisson_noise_model}
}


\end{document}